\documentclass[a4paper,reqno]{amsart}

\usepackage{setspace}
\usepackage[totalwidth=14cm,totalheight=22cm]{geometry}
\usepackage[T1]{fontenc}
\usepackage[dvips]{graphicx}
\usepackage{epsfig}
\usepackage{amsmath,amssymb,amscd,amsthm}
\usepackage{subfigure}
\usepackage[latin1]{inputenc}
\usepackage{latexsym}
\usepackage{graphics}
\usepackage{colortbl}
\usepackage{color}
\usepackage{ae}
\usepackage{enumitem}
\usepackage[all]{xy}
\usepackage{scalerel}
\usepackage{tikz}
\usepackage{cancel}
\usepackage{bbm,amsbsy}
\usepackage{mathrsfs}

    \graphicspath{{../}}
    \makeatletter
    \providecommand*{\input@path}{}
    \g@addto@macro\input@path{{../}}
    \makeatother
    \newcommand{\red}{\textcolor{red}} %
    \newcommand{\blue}{\textcolor{blue}} %
    \newcommand{\magenta}{\textcolor{magenta}} 
    \usepackage[normalem]{ulem}

\newlist{steps}{enumerate}{1}
\setlist[steps, 1]{itemsep=3pt,leftmargin=0cm,itemindent=.5cm,labelwidth=\itemindent,labelsep=0cm,align=left,label = \emph{Step \arabic*}:\,}

\makeatletter
\newtheorem*{rep@theorem}{\rep@title}
\newcommand{\newreptheorem}[2]{%
\newenvironment{rep#1}[1]{%
 \def\rep@title{#2 \ref{##1}}%
 \begin{rep@theorem}}%
 {\end{rep@theorem}}}
\makeatother

\makeatletter
\newtheorem*{rep@cor}{\rep@title}
\newcommand{\newrepcor}[2]{%
\newenvironment{rep#1}[1]{%
 \def\rep@title{#2 \ref{##1}}%
 \begin{rep@cor}}%
 {\end{rep@cor}}}
\makeatother

\makeatletter
\newtheorem*{rep@prop}{\rep@title}
\newcommand{\newrepprop}[2]{%
\newenvironment{rep#1}[1]{%
 \def\rep@title{#2 \ref{##1}}%
 \begin{rep@prop}}%
 {\end{rep@prop}}}
\makeatother

\newtheorem{cor}{Corollary}[section]
\newtheorem{corx}{Corollary}
\renewcommand{\thecorx}{\Alph{corx}}

\newtheorem{theorem}[cor]{Theorem}
\newtheorem{thmx}[corx]{Theorem}
\renewcommand{\thethmx}{\Alph{thmx}}

\newtheorem{prop}[cor]{Proposition}

\renewcommand{\thepropx}{\Alph{propx}}

\newrepcor{cor}{Corollary}
\newreptheorem{theorem}{Theorem}
\newrepprop{prop}{Proposition}

\newtheorem{lemma}[cor]{Lemma}

\newtheorem*{problem*}{Problem}

\theoremstyle{definition}
\newtheorem{defi}[cor]{Definition}
\theoremstyle{remark}
\newtheorem{remark}[cor]{Remark}
\newtheorem*{remark*}{Remark}
\newtheorem{example}[cor]{Example}

\theoremstyle{plain}
\newcommand{\thistheoremname}{}
\newtheorem*{genericthm}{\thistheoremname}
\newenvironment{namedthm}[1]
  {\renewcommand{\thistheoremname}{#1}%
   \begin{genericthm}}
  {\end{genericthm}}

\newcommand{\rank}{\hbox{rank}\ }
\newcommand{\Stab}{\hbox{Stab}\ }

\DeclareMathOperator\arctanh{arctanh}
\DeclareMathOperator*{\bigcdot}{\scalerel*{\cdot}{\bigodot}}
 
\newcommand{\dwp}{d_{\mathrm{W}\!\mathrm{P}}}
\newcommand{\weil}{{\mathrm{W}\!\mathrm{P}}}
\newcommand{\Th}{{\mathrm{T}\mathrm{h}}}
\newcommand{\dl}{d_{\mathrm{L}}}
\newcommand{\Tt}{\mathbb{T}}
\newcommand{\cC}{{\mathcal C}}
\newcommand{\cD}{{\mathcal D}}
\newcommand{\cE}{{\mathcal E}}
\newcommand{\cF}{{\mathcal F}}
\newcommand{\cL}{{\mathcal L}}
\newcommand{\cM}{{\mathcal M}}
\newcommand{\cN}{{\mathcal N}}
\newcommand{\cP}{{\mathcal P}}
\newcommand{\cS}{{\mathcal S}}
\newcommand{\cT}{{\mathcal T}}
\newcommand{\cQ}{{\mathcal Q}}
\newcommand{\cW}{{\mathcal W}}
\newcommand{\cY}{{\mathcal Y}}
\newcommand{\cDY}{{\mathcal{DY}}}
\newcommand{\cCP}{{\mathcal C\mathcal P}}
\newcommand{\cML}{{\mathcal M\mathcal L}}
\newcommand{\cFML}{{\mathcal F\mathcal M\mathcal L}}
\newcommand{\cGH}{{\mathcal G\mathcal H}}
\newcommand{\cQF}{{\mathcal Q\mathcal F}}
\newcommand{\cQS}{{\mathcal Q\mathcal S}}
\newcommand{\C}{{\mathbb C}}
\newcommand{\U}{{\mathbb U}}
\newcommand{\Hh}{{\mathbb H}}
\newcommand{\PP}{\mathbb{P}}
\renewcommand{\Pr}{\mathbb{P}}
\newcommand{\N}{{\mathbb N}}
\newcommand{\R}{{\mathbb R}}
\newcommand{\Z}{{\mathbb Z}}
\newcommand{\Kt}{\tilde{K}}
\newcommand{\Mt}{\tilde{M}}
\newcommand{\dr}{{\partial}}
\newcommand{\betab}{\overline{\beta}}
\newcommand{\kappab}{\overline{\kappa}}
\newcommand{\pib}{\overline{\pi}}
\newcommand{\taub}{\overline{\tau}}
\newcommand{\ub}{\overline{u}}
\newcommand{\Sigmab}{\overline{\Sigma}}
\newcommand{\alphad}{\dot{\alpha}}
\newcommand{\gd}{\dot{g}}
\newcommand{\hd}{\dot{h}}
\newcommand{\diff}{\mbox{Diff}}
\newcommand{\dev}{\mbox{dev}}
\newcommand{\devb}{\overline{\mbox{dev}}}
\newcommand{\devt}{\tilde{\mbox{dev}}}
\newcommand{\vol}{\mbox{Vol}}
\newcommand{\hess}{\mbox{Hess}}
\newcommand{\cb}{\overline{c}}
\newcommand{\gb}{\overline{g}}
\newcommand{\Mb}{\overline{M}}
\newcommand{\db}{\overline{\partial}}
\newcommand{\Sigmat}{\tilde{\Sigma}}
\newcommand{\Hyp}{\mathbb{H}}
\newcommand{\Sph}{\mathbb{S}}
\newcommand{\AdS}{\mathbb{A}\mathrm{d}\mathbb{S}}
\newcommand{\dS}{\mathrm{d}\mathbb{S}}
\newcommand{\cbull}{\bullet}
\newcommand{\psl}{\mathfrak{sl}}
\newcommand{\SL}{\mathrm{SL}}
\newcommand{\PSL}{\mathrm{PSL}}
\newcommand{\dual}{\star}
\newcommand{\ph}{\varphi}
\newcommand{\sg}{\varsigma}
\newcommand{\im}{\textrm{Im}}
\newcommand{\co}{\mathrm{co}\,}
\newcommand{\injrad}{\mathrm{injrad}}
\newcommand{\arcsinh}{\mathrm{arcsinh}}
\newcommand{\thin}{\mathrm{thin}}
\newcommand{\thick}{\mathrm{thick}}
\newcommand{\Isom}{\mathrm{Isom}}

\newcommand{\cunc}{{C}^\infty_c}
\newcommand{\cun}{{\mathrm{C}}^\infty}
\newcommand{\dd}{d_D}
\newcommand{\dmin}{d_{\mathrm{min}}}
\newcommand{\dmax}{d_{\mathrm{max}}}
\newcommand{\Dom}{\mathrm{Dom}}
\newcommand{\dn}{d^\nabla}
\newcommand{\ded}{\delta_D}
\newcommand{\delmin}{\delta_{\mathrm{min}}}
\newcommand{\delmax}{\delta_{\mathrm{max}}}
\newcommand{\hmin}{H_{\mathrm{min}}}
\newcommand{\maxi}{\mathrm{max}}
\newcommand{\oL}{\overline{L}}
\newcommand{\oP}{{\overline{P}}}
\newcommand{\xb}{{\overline{x}}}
\newcommand{\yb}{{\overline{y}}}
\newcommand{\Ran}{\mathrm{Ran}}
\newcommand{\tgamma}{\tilde{\gamma}}
\newcommand{\gammab}{\overline{\gamma}}
\newcommand{\Diffeo}{\mbox{Diff}}
\newcommand{\cotan}{\mbox{cotan}}
\newcommand{\area}{\mbox{Area}}
\newcommand{\lambdat}{\tilde\lambda}
\newcommand{\xt}{\tilde x}
\newcommand{\Ct}{\tilde C}
\newcommand{\St}{\tilde S}
\newcommand{\trace}{\mbox{\rm tr}}
\newcommand{\tr}{\mbox{\rm tr}}
\newcommand{\tgh}{\mbox{th}}
\newcommand{\sh}{\mathrm{sinh}\,}
\newcommand{\ch}{\mathrm{cosh}\,}
\newcommand{\grad}{\operatorname{grad}}
\newcommand{\isom}{\mathrm{Isom}}
\newcommand{\hc}{H\C^3}
\newcommand{\cp}{\C P^1}
\newcommand{\rp}{\R P^1}
\newcommand{\Vol}{\mathrm{Vol}}
\newcommand{\diam}{\mathrm{diam}}
\newcommand{\Area}{\mathrm{Area}}
\newcommand{\dth}{d_{\mathrm{Th}}}
\newcommand{\dteich}{d_{\mathrm{Teich}}}
\newcommand{\h}{\mathbb{H}}
\newcommand{\Ker}{\mathrm{Ker}}
\newcommand{\Ima}{\mathrm{Im}}

\newcommand{\Ip}{\mathrm{I}^+}
\newcommand{\Ipm}{\mathrm{I}^-}

\newcommand{\note}[1]{{\small {\color[rgb]{1,0,0} #1}}}

\newcommand{\rar}{\rightarrow}
\newcommand{\lra}{\longrightarrow}
\newcommand{\war}{\rightharpoonup}

\newcommand{\da}{\omega}

\newcommand{\id}{\mathrm{I}}
\newcommand{\SO}{\mathrm{SO}}
\newcommand{\so}{\mathfrak{so}}
\newcommand{\ddt}{\left.\frac{d}{dt}\right|_{t=0}}
\newcommand{\cA}{\mathcal{A}}
\newcommand{\D}{\mathbb{D}}

\newcommand{\I}{\mathrm{I}}
\newcommand{\II}{\mathrm{I}\hspace{-0.04cm}\mathrm{I}}
\newcommand{\III}{\mathrm{I}\hspace{-0.04cm}\mathrm{I}\hspace{-0.04cm}\mathrm{I}}

\renewcommand{\Re}{\mathrm{Re}}

\def\l{\lambda}

\def\Hess{\mathrm{Hess}}
\def\Div{\mathrm{div}}
\def\Teich{\mathcal{T}}
\def\Grad{\mathrm{grad}}
\def\wti#1{\widetilde{#1}}
\def\what#1{\widehat{#1}}
\def\ol#1{\overline{#1}}

\def\earth{\mathcal E}
\def\En{\mathbb{I}}
\def\qd{\varphi}
\def\qdd{q}
\def\length{\ell}
\def\ed{M}
\def\Id{\mathbbm 1}

\DeclareRobustCommand{\SkipTocEntry}[4]{}

\begin{document}

\setcounter{secnumdepth}{3}
\setcounter{tocdepth}{2}

\title[Entire surfaces of constant curvature in Minkowski 3-space]{Entire surfaces of constant curvature\\in Minkowski 3-space}

\author[Francesco Bonsante]{Francesco Bonsante}
\address{Francesco Bonsante: Dipartimento di Matematica ``Felice Casorati", Universit\`{a} degli Studi di Pavia, Via Ferrata 5, 27100, Pavia, Italy.} \email{bonfra07@unipv.it} 
\author[Andrea Seppi]{Andrea Seppi}
\address{Andrea Seppi: University of Luxembourg, Mathematics Research Unit, Maison du Nombre, 6 Avenue de la Fonte, Esch-sur-Alzette L-4364 Luxembourg. } \email{andrea.seppi@uni.lu}
\author[Peter Smillie]{Peter Smillie}
\address{Peter Smillie: Harvard University, Department of Mathematics, One Oxford St, Cambridge, MA, 02138, USA.} \email{smillie@math.harvard.edu}


\thanks{The first aurthor was partially supported by Blue Sky Research project ``Analytic and geometric properties of low-dimensional manifolds" . The first two authors are members of the national research group GNSAGA}

\begin{abstract}
This paper concerns the global theory of properly embedded spacelike surfaces in three-dimensional Minkowski space in relation to their Gaussian curvature. We prove that every regular domain which is not a wedge is uniquely foliated by properly embedded convex surfaces of constant Gaussian curvature. This is a consequence of our classification of surfaces with bounded prescribed Gaussian curvature, sometimes called the Minkowski problem, for which partial results were obtained by Li, Guan-Jian-Schoen, and Bonsante-Seppi. Some applications to minimal Lagrangian self-maps of the hyperbolic plane are obtained.
\end{abstract}

\maketitle
%

\addtocontents{toc}{\SkipTocEntry}

\section*{Introduction} 
Minkowski $3$-space is the simply connected geodesically complete flat Lorentzian manifold $\R^{2,1}=(\R^3, dx_1^2+dx_2^2 - dx_3^2)$.
A $C^1$ immersed surface $\Sigma$ in $\R^{2,1}$ is called \emph{spacelike} if the restriction of the Lorentzian metric to $T\Sigma$ is a Riemannian metric.
Any spacelike surface is locally a graph of the form $x_3=f(x_1,x_2)$  for some function $f\in C^1(\R^2)$
which is strictly $1$-Lipschitz with respect to the Euclidean metric of the plane.
The aim of the paper is to provide a full classification of properly embedded spacelike surfaces with constant Gaussian curvature (CGC) in Minkowski space
in terms of their asymptotic behavior.

Let us explain more precisely the content of the classification.
Given a properly embedded spacelike surface $\Sigma$ in $\R^{2,1}$, its domain of dependence $\mathcal D_\Sigma$ 
is the region of $\R^{2,1}$ consisting of those points through which any inextendable causal path must meet $\Sigma$ (Definition \ref{def domain of dependence}).
We show in Corollary \ref{cor domains dependence} that the domain of dependence of a properly embedded CGC surface $\Sigma$ is a \emph{regular domain} up to time-reversal. 
Here the terminology is taken from \cite{Bonsante}: a regular domain is a convex open domain in 
$\R^{2,1}$ that is the intersection of the future of its null support planes and is neither the whole $\mathbb R^{2,1}$ nor the future of a single null plane.

Among regular domains we call \emph{wedges} those domains which are obtained as the intersection of the futures of exactly two null planes.
It turns out that a wedge is never the domain of dependence of a properly embedded CGC surface.
The main goal of this paper is to prove that aside from this case, every regular domain is the domain of dependence of exactly one
properly embedded surface of constant Gaussian curvature $K$, for any fixed $K\in(0,+\infty)$. In this paper, $K$ is the \emph{extrinsic} Gaussian curvature, which is the determinant of the shape operator and the negative of the intrinsic Gaussian curvature.

\begin{thmx} \label{thm-CGC-problem}
Fix $K>0$. Given any regular domain $\mathcal D\subset \R^{2,1}$ which is not a wedge, there exists a unique properly embedded CGC-$K$ surface whose domain of dependence is $\mathcal D$. \end{thmx}

Once we conclude as a consequence of Corollary \ref{cor domains dependence} that the domain of dependence of every future-convex CGC-$K$ surface is a regular domain and not a wedge, we immediately have the corollary:


\begin{corx} \label{thm-CGC-problem-2}
Fix $K>0$. There is a bijection from the set of properly embedded future-convex CGC $K$-surfaces in $\R^{2,1}$ to the set of regular domains which are not a wedge, given by associating to a surface its domain of dependence.
\end{corx}


We may restate our main theorem in terms of lower semi-continuous functions on the circle. In the Klein model of the hyperbolic plane $\Hyp^2$, points on the boundary of $\Hyp^2$ represent lightlike directions in $\R^{2,1}$, which by the duality induced by the Lorentzian inner product
are in bijection with null linear planes. The space of null affine planes in $\R^{2,1}$ is naturally
identified to a cylinder $\partial\Hyp^2\times\R$. 
Two points $\partial\Hyp^2\times\R$  correspond to parallel planes if and only if their first components coincide.
From this point of view regular domains are in bijection with lower-semicontinuous functions $\ph:\partial\Hyp^2\to\R\cup\{+\infty\}$
that are finite at least on two points (Proposition \ref{prop:bijection reg dom}). We will call $\mathcal{D}_\ph$ the domain corresponding to the function $\ph$. If $\Sigma$ is the graph of an entire convex function $f:\R^2\to\R$, we call $\Sigma$ \emph{entire}. A simple argument (Proposition \ref{prop properly emb iff entire}) shows that every properly immersed spacelike surface is entire. It was proved in \cite[Subsection 2.3]{Bonsante:2015vi} that the function $\ph$ corresponding to
$\mathcal D_\Sigma$ is given by
\[
     \ph(\mathsf{y})=-\lim_{r\to+\infty}(f(r\mathsf{y})-r)~.
\]
In this way the function $\ph$ encodes the asymptotic behavior of the surface $\Sigma$. The graph of $\ph$ can also be regarded
as the asymptotic boundary of $\Sigma$ in the Penrose causal compactification of $\R^{2,1}$, but this point of view will not developed here.

Therefore Theorem \ref{thm-CGC-problem} establishes a correspondence between entire CGC graphs and lower semi-continuous functions on the circle which may take the value $+ \infty$:

\begin{corx} \label{cor classification Ksurfaces}
Fix $K>0$. There is a bijection between the set of future-convex entire surfaces of constant Gaussian curvature $K$ in $\R^{2,1}$ and the set of lower semicontinuous functions $\ph:\partial\Hyp^2\to\R\cup\{+\infty\}$ finite on at least three points. 
\end{corx}

Next, we will prove that any regular domain that is not a wedge is foliated by CGC surfaces with constant Gaussian curvature ranging from 0 to $\infty$:

\begin{thmx} \label{thm CGC foliation}
For every regular domain $\mathcal D$ in $\R^{2,1}$ which is not a wedge, there exists a unique foliation by properly embedded CGC $K$-surfaces, with $K\in(0,\infty)$.
\end{thmx}
\noindent
As a result, the function $\tau = K^{-1}$ gives a canonical proper function with time-like gradient on every regular domain $\mathcal D$. It is an example of a \emph{geometric canonical time function}, called the $K$-time in \cite{barbotzeghib}.

The study of CGC surfaces in Minkowski space goes back at least to Hano and Nomizu \cite{hanonomizu} who first pointed out the existence of non-standard isometric immersions
of $\Hyp^2$ in $\R^{2,1}$. In \cite{Li} An-Min Li proved the existence part of Theorem \ref{thm-CGC-problem} and 
Corollary \ref{cor classification Ksurfaces} in the case that $\ph$ is smooth.
This result was improved by Guan, Jian and Schoen \cite{schoenetal}: they proved  the existence of an entire CGC $K$-surface only assuming $\ph$ is Lipschitz and possibly infinite on a single open arc. In another direction, Barbot, B\'eguin and Zeghib proved  in  \cite{barbotzeghib} that any regular domain invariant by an affine deformation of a uniform lattice in $\SO(2,1)$ 
contains a CGC $K$-surface. In \cite{Bonsante:2015vi}  the first two authors proved the existence of a CGC surface in a given regular domain under the assumption that the corresponding function $\ph$  is lower semi-continuous and bounded. Moreover in that work it was proved that entire CGC surfaces with bounded second fundamental form are in bijection with regular domains whose corresponding function is Zygmund continuous.

In higher dimensions the problem can be generalized in different ways.
Li's original theorem applies to  
hypersurfaces of constant extrinsic curvature in any dimension. However in dimensions greater than 3 the smoothness condition on $\ph$ plays an 
important role. In fact  an example has been pointed out in \cite{bonfill} of an affine deformation of a uniform lattice in $\SO(3,1)$ which preserve
no hypersurface in $\R^{3,1}$ with constant extrinsic curvature. By contrast in \cite{grahamsmith} it has been proved that any affine deformation
of a uniform lattice in $\SO(3,1)$ preserves exactly one hypersurface of constant scalar curvature. 

Theorem \ref{thm-CGC-problem} has been obtained as a consequence of  more general statements about  properly embedded spacelike 
surfaces in Minkowski 3-space of positive Gaussian curvature. Recall that there is a natural notion of Gauss map for spacelike surfaces in Minkowski space. In this context the Gauss map takes values on the hyperbolic plane, which is identified with the set of future-directed unit timelike vectors.
We first prove that if the Gaussian curvature of $\Sigma$ is bounded  by two positive constants  then the image of the Gauss map is a domain
of $\Hyp^2$ bounded by geodesics. More specifically, by Corollary \ref{cor domains dependence}, the domain of dependence of $\Sigma$ is of the form $\mathcal{D}_\ph$ for some lower semi-continuous function $\ph: \partial \Hyp^2 \to \R \cup \{+\infty\}$.
\begin{thmx}\label{thm gauss image}
Let $\Sigma$ a properly embedded spacelike surface in $\mathbb R^{2,1}$ with Gaussian curvature  bounded from above and below by positive constants. Let $\ph:\partial\Hyp^2\to\mathbb R\cup\{+\infty\}$
be such that the domain of dependence of $\Sigma$ is $\mathcal{D}_\ph$. Then the Gauss map of $\Sigma$ is a diffeomorphism onto the interior of the convex hull of
$\partial\mathbb H^2\setminus\ph^{-1}(+\infty)$. 
 \end{thmx}

In Section \ref{sec image of gauss map and support function on boundary} we will give a more precise version of this result, see Theorem \ref{prop constant curvature convex hull2}.
Notice in particular that the image of the Gauss map of a surface with Gaussian curvature bounded by two positive constants 
depends only on the domain of dependence of $\Sigma$. 
We will denote by $\Omega_{\ph}$ the interior of the convex hull of $\partial\Hyp^2\setminus\ph^{-1}(+\infty)$. 
The second general result which we achieve, and which implies Theorem \ref{thm-CGC-problem}, concerns the Minkowski problem. In general 
the Minkowski problem asks for a convex surface $\Sigma$  in $\R^{2,1}$ for which the domain $\Omega_\Sigma:=G_\Sigma(\Sigma)$
and  the function $\psi :=\kappa_\Sigma\circ G_{\Sigma}^{-1}:\Omega_{\Sigma}\to \R_{>0}$ are prescribed.
We will prove the following statement: 

\begin{thmx} \label{thm mink problem}
Let  $\mathcal D$ be a regular domain in $\R^{2,1}$ which is not a wedge, defined by a function $\ph:\partial\Hyp^2\to\mathbb R\cup\{+\infty\}$,
and let $\psi$  be a continuous function defined on $\Omega_{\ph}$
 which is bounded by two positive constants.
 There exists a {unique} entire spacelike surface $\Sigma$ in $\mathcal D$ whose domain of dependence is $\mathcal D$ and whose curvature function satisfies:
$$\kappa(p)=\psi\circ G_\Sigma(p)$$
for every $p\in\Sigma$, where $G_\Sigma$ is the Gauss map of $\Sigma$.
\end{thmx}

Finally, we give an application of Theorem \ref{thm gauss image} to minimal Lagrangian maps between hyperbolic surfaces. 
The Gauss map of a CGC isometric immersion with $K=1$ into $\mathbb R^{2,1}$ is minimal Lagrangian: this means that it is area preserving and its graph
is a minimal surface in the product. Conversely if $F:\Sigma\to\mathbb H^2$ is a minimal Lagrangian map with $\Sigma$ hyperbolic, 
one can produce an isometric immersion $\sigma_F:\Sigma\to\mathbb R^{2,1}$ such that $F$ coincides with the Gauss map of
$\sigma_F$.
Theorem \ref{thm gauss image} states that if $\sigma_F$ is a proper embedding, then $F$ is injective and its image is a domain bounded by geodesics.
As $\sigma_F$ is always a proper embedding if the domain is complete, we get the following corollary:
  
\begin{corx} \label{cor image min lag map} 
Let $F:\Hyp^2\to\Hyp^2$ be a minimal Lagrangian map. Then $F$ is a diffeomeorphism onto the interior of the convex hull of 
$\overline{F(\Hyp^2)}\cap \partial\Hyp^2$.
\end{corx}



\subsection*{Strategy of the proofs} \addtocontents{toc}{\SkipTocEntry}

The \emph{support function} $u_\Sigma$ (Definition \ref{defi support functions}) of a surface $\Sigma$ is a closed convex function (Definition \ref{def closed covex}) defined on the closed unit disk $\D$, the Klein model of $\overline{\mathbb{H}^2}$. If $\Sigma$ is properly embedded and has positive Gaussian curvature $\kappa(p)$ at every point  $p$, then we show that the Gauss map $G_\Sigma$ is injective, $u_\Sigma$ is finite on the image of $G_\Sigma$.  Moreover on this image $u_\Sigma$ satisfies the equation
\begin{equation} \label{eqn: minkowski problem monge-ampere}
\det D^2 u_\Sigma (\mathsf{x}) =\frac{1}{\kappa(G^{-1}_\Sigma(\mathsf{x}))}(1-|\mathsf x|^2)^{-2}
\end{equation}
where  $|\mathsf x|$ is the Euclidean norm of $\mathsf x$ in the disk \cite{Li}.
The function $\ph$ defining the domain of dependence of $\Sigma$  coincides with the restriction of $u_\Sigma$ to $\partial\mathbb H^2$. 

The support function $u_\Sigma$ determines the surface $\Sigma$.
In this way, Theorem \ref{thm mink problem} can be interpreted as the existence and uniqueness of solutions to a generalized Dirichlet problem for equation (\ref{eqn: minkowski problem monge-ampere}) with boundary data given by $\ph$. It differs from the standard Dirichlet problem in that the boundary data $\ph$ and the solution $u$ may both take the value $+ \infty$. At the same time, the condition that $\Sigma$ be entire restricts the class of functions $u$ we consider to those that are \emph{gradient surjective} (Definition \ref{defi gradient surjective}).

This problem is made tractable by Theorem \ref{thm gauss image} and its more precise form Theorem \ref{prop constant curvature convex hull2}. These theorems allow us to reduce our generalized Dirichlet problem on $\D$ to a problem on the smaller domain $\Omega_\ph$, on which $u$ is necessarily finite.



The idea to prove Theorem \ref{thm gauss image}, or more generally Theorem \ref{prop constant curvature convex hull2},
is to use a barrier argument: if $\Sigma$ is a surface with Gaussian curvature bounded by two positive constants,
then for every boundary chord $c$  of $\Omega_\ph$ we produce a convex surface $\Sigma'$ so that
\begin{itemize}
\item $\Sigma'$ lies above $\Sigma$;
\item the image of the Gauss map of  $\Sigma'$ does not contain any point in the half plane $H$ bounded by  $c$  in the complement of $\Omega_{\ph}$.
\end{itemize}
The first point implies that the image of the Gauss map of $\Sigma$ is contained in the image of the Gauss map of $\Sigma'$ so that the second point shows that
$G_\Sigma(\Sigma)$ does not contain any point in $H$.

An important ingredient in the proof is the comparison principle for Monge-Amp\`ere equations. 
However we have to apply the comparison principle to functions that are in general unbounded.
So we need to prove a refined version of the comparison principle
for convex functions that are possibly infinite at some points, which we do in Proposition \ref{prop:genmax}. 
Here the hypothesis that $\Sigma$ is properly embedded plays a key role.

Once we have reduced Theorem \ref{thm mink problem} to a problem on $\Omega_\ph$, we are able to produce a solution. But in order to prove that the corresponding CGC surface is entire, we need another barrier. Specifically, from the point of view of the surfaces themselves we need a lower barrier, or from the dual point of view of support functions we need an upper barrier.
The general strategy follows the same line as in \cite{Bonsante:2015vi, schoenetal}. However, those papers use upper barriers invariant under a $1$-parameter group, whereas such surfaces can never provide upper barriers for the general class of boundary values $\ph$ that we consider. The support function of a barrier which is invariant under a one-parameter group must have boundary values which are finite on at least an open interval, whereas we consider functions $\ph$ that are finite on as few as three points.
Therefore we construct in Section \ref{sec triangular surfaces} an entire CGC-$K$ for which $\ph$ is finite at exactly three points, i.e. one whose domain of dependence is the intersection of the future of three null planes.
This surface and the refined comparison principle are the key new ingredients to prove Theorem \ref{thm mink problem}.

The construction of this particular surface is based on the harmonic maps $f_\pm:\mathbb C\to\mathbb H^2$ with Hopf differential $\pm zdz^2$.
 It is known that the images of $f_\pm$ are open ideal triangles $T_\pm$ \cite{hantamtreibergswan}. The map $F=f_+\circ f_-^{-1}:T_-\to\mathbb H^2$ is minimal Lagrangian
 and one studies the corresponding embedding $\sigma_F:T_-\to\mathbb R^{2,1}$. 
The embedding data of this surface are explicitly described in terms of the Hopf differential and the holomorphic energy of the harmonic map.
The technical part is to show that the corresponding surface is properly embedded. Using the symmetry of the embedding $\sigma_F$ we reduce the problem to
showing that the image of a line of symmetry is a properly embedded curve in Minkowski space. This is finally proved by studying  the growth of the holomorphic energy
of the harmonic map along the curve and its relation with the principal curvature of this isometric immersion.
Once the barrier (which we will call a \emph{triangular} surface) is fully described, the Minkowski problem is considered. 

For a given lower semicontinuous function $\ph:\partial\mathbb H^2\to\mathbb R\cup\{+\infty\}$ and a bounded continuous 
 function $\psi$ defined on the interior $\Omega_\ph$ of the convex hull 
of $\ph^{-1}(\mathbb R)$ we construct a function $u$ on $\overline{\Omega_\ph}$ which solves the equation
\begin{equation}\label{eq:monge ampere}
   \det D^2 u (\mathsf{x}) =\frac{1}{\psi(\mathsf{x})}(1-|\mathsf x|^2)^{-2}
\end{equation}
and is the linear interpolation of $\ph$ on the boundary of $\Omega_\ph$.
To the end we consider the convex envelope $\mathrm{conv}(\ph)$  of $\ph$. Taking an interior approximation $\Omega_n$ of $\Omega_\ph$ by convex domains, we consider the solution $u_n$ 
of  the equation \eqref{eq:monge ampere} 
on $\Omega_n$  with boundary data $\mathrm{conv}(\ph)$. Applying the comparison principle with classical barriers we prove that $u_n$ converges 
to the solution of the problem. The function $u$ defines a spacelike convex surface $\Sigma$ in Minkowski space, that must be proved to be entire.
More precisely $\Sigma$ is part of an entire achronal surface, which however could contain some additional regions which are not strictly convex. 
 The problem reduces to showing that $\overline\Sigma$ does not meet any plane $P$ whose normal vector lies in $\partial \Omega_\ph$. 
 In fact for such a $P$ 
 we will show that there is a triangular surface which separates $\Sigma$ from $P$. The proof of this fact is based again on the comparison principle.


\subsection*{Organization of the paper} \addtocontents{toc}{\SkipTocEntry}


Sections \ref{sec:splk}, \ref{analytical sec}, and \ref{sec monge ampere} contain preliminaries as well as proofs of some general theorems for which we could not find references. In Section \ref{sec:splk} we quickly review the theory of spacelike surfaces in Minkowski space. First and second fundamental forms are introduced
and the relevant Gauss-Codazzi equations explained. We show that properly embedded spacelike surfaces are graph, and introduce the notion of domain 
of dependence. We will see that aside from few exceptions  the domain of dependence is a regular domain.
In Section \ref{analytical sec}, the notion of support function is given and the relation between the boundary
value of the support function and the domain of dependence of the surface is  pointed out. The relation between curvature of the surface and the support function
is described, and Minkowski problem is shown to be equivalent to a Dirichlet problem for a Monge-Amp\`ere equation.
In Section \ref{sec monge ampere} 
we describe the analytical tools we will need to solve our problem. Classical results of stability for solutions of Monge-Amp\`ere equations are
given and a refined version of the comparison principle for unbounded functions is proved.

The remainder of the paper contains our main results on CGC surfaces. In Section \ref{sec image of gauss map and support function on boundary} we prove Theorem \ref{thm gauss image}. First we study some special CGC surfaces whose domain of dependence is the future
of a spacelike half-line in Minkowski space. Those surfaces and our comparison principle  will be the key ingredients to prove Theorem \ref{thm gauss image}.
In Section \ref{sec triangular surfaces} we study the triangular  surfaces. First we construct the embedding data of a CGC immersion 
on $\mathbb C$ by means of a correspondence with harmonic maps and minimal Lagrangian maps. Then we prove that 
this immersion is a proper embedding. Section \ref{sec existence and uniqueness} is devoted to solving the Minkowski problem. As an application we will prove in
 Section \ref{sec:foliation}  that any regular domain $\mathcal D$ which is not a wedge is foliated by CGC surfaces.
 
 Finally in Section \ref{sec:final} we point out an open question.
 
 \subsection*{Acknowledgements} We thank Jean-Marc Schlenker for his interest and encouragement throughout. The third author also wishes to thank Shing-Tung Yau for inspiring interest in the problem.

\section{Spacelike surfaces in Minkowski space}\label{sec:splk}

Minkowski $(2+1)$-space is the simply connected geodesically complete flat Lorentzian manifold $\R^{2,1}=(\R^3, dx_1^2+dx_2^2 - dx_3^2)$.
A nonzero tangent vector $\pmb v$ is called \emph{spacelike, lightlike} or \emph{timelike} if $\langle \pmb v,\pmb v\rangle>0$, $\langle \pmb v,\pmb v\rangle=0$ or $\langle \pmb v,\pmb v\rangle<0$ respectively. 
We also say $\pmb v$ is \emph{causal} if it is either lightlike or timelike, and $\pmb v$ is \emph{achronal} if it is either lightlike or spacelike. A causal vector is either \emph{future-directed} if its $x_3$-component is positive and \emph{past-directed} if its $x_3$-component is negative. 

A point $\pmb p$ is in the \emph{future} of $\pmb q$ (and $\pmb q$ is in the \emph{past} of $\pmb p$) if $\pmb p-\pmb q$ is timelike future-directed. We denote by $I^+(\pmb p)$ (resp. $I^-(\pmb p)$) the open cone of points in the future (resp. past) of $\pmb p$. If $S$ is any set in Minkowski space, we then define the \emph{future} and \emph{past} of $S$ as
$$I^+(S)=\bigcup_{\pmb p\in S}I^+(\pmb p)\qquad\text{and}\qquad I^-(\Sigma)=\bigcup_{\pmb p\in S}I^-(\pmb p) $$
and we say $S$ is future-complete if $I^+(S) \subset S$.

A $C^0$ submanifold $\Sigma$ is causal (resp. achronal) if for each point $\pmb p \in \Sigma$, there is a neighborhood of $\pmb p$ in which every point of $\Sigma$ is causally (resp. achronally) separated from $\pmb p$. For some of the preliminaries we allow immersed surfaces, in which case ``locally'' means locally in the domain; however for the bulk of the paper we are concerned only with entire surfaces, which are necessarily embedded. 
A $C^1$ surface is \emph{spacelike}, \emph{lightlike}, or \emph{timelike} if the induced metric on the tangent space is positive definite, degenerate, or indefinite respectively. If $\Sigma \subset \R^{2,1}$ is a $C^1$ spacelike surface, the \emph{future unit normal vector field} is the unique future-directed vector field $\pmb n$ orthogonal to $\Sigma$ such that $\langle \pmb n,\pmb n\rangle=-1$.

The purpose of the following section is to introduce preliminary geometric notions on spacelike surfaces in $\R^{2,1}$, including the definition of \emph{entire surface} and of \emph{domain of dependence}, and finally state the Minkowski and CGC problems which are the main focus of this paper.

\subsection{Embedding data for spacelike surfaces}

Let us denote by $D$ the flat connection of $\R^3$. For a smoothly immersed spacelike surface $\Sigma$ in $\R^{2,1}$ we recall:
\begin{itemize}
\item The \emph{first fundamental form} $\I$ is the Riemannian metric on $T \Sigma$ given by the restriction of the metric $\langle\cdot,\cdot\rangle$.
\item The \emph{Levi-Civita connection $\nabla$} and \emph{second fundamental form $\II$} are defined on $T\Sigma$ as the tangential and normal components respectively of the connection $D$:
$$D_{\pmb v} \pmb w = \nabla_{\pmb v} \pmb w + \II(\pmb v,\pmb w)\pmb n.$$
\item The \emph{shape operator} $B$ is the self-adjoint endomorphism of $T \Sigma$ given by differentiating the normal vector field $\pmb n$:
$$B(\pmb v)= D_{\pmb v}(\pmb n).$$
\end{itemize}


The three objects $\I$, $\II$, and $B$ are related by the \emph{Weingarten equation} $\II(\pmb v,\pmb w)=\I(B(\pmb v),\pmb w)$. The \emph{third fundamental form} $\III$ is defined by $\III(v,w) = \I(B(v),B(w))$.
Moreover, the pair $(\I,B)$ satisfies the \emph{Gauss equation}:
\begin{equation}  \label{eq gauss}
\kappa_I=-\det B~,
\end{equation}
where  $\kappa_I$ is the intrinsic curvature of $\I$, and the \emph{Codazzi equation}:
\begin{equation}  \label{eq codazzi}
d^{\nabla}B=0~,
\end{equation}
where $d^{\nabla}$ is the extension of $\nabla$ to $T\Sigma$-valued differential forms, which is given by the formula (equivalent to the vanishing of the torsion of $\nabla$):
$$d^{\nabla}B(\pmb v,\pmb w)=(\nabla_{\pmb v} B)(\pmb w)-(\nabla_{\pmb w} B)(\pmb  v)~.$$

The Fundamental Theorem of surface theory, in the case of Minkowski space, shows that Equations \eqref{eq gauss} and \eqref{eq codazzi} also provide sufficient conditions to determine, at least for a simply connected surface $\Sigma$, a spacelike immersion into $\R^{2,1}$:

\begin{theorem} \label{fund theorem}
Let $\Sigma$ be a simply connected surface. Given a Riemannian metric $\I$ on $\Sigma$ and a (1,1)-tensor $B\in\Gamma(\mathrm{End}(T\Sigma))$, self-adjoint for $\I$, such that the pair $(\I,B)$ satisfies Equations \eqref{eq gauss} and \eqref{eq codazzi}, there exists a spacelike immersion $\sigma:\Sigma\to\R^{2,1}$ such that the pull-back of the first fundamental form and shape operator of $\sigma(\Sigma)$ coincide with $\I$ and $B$. Moreover, any two such immersions differ by post-composition with a global isometry of $\R^{2,1}$.
\end{theorem}

We define the Gaussian curvature in an extrinsic way:


\begin{defi} \label{defi HKsurface}
The \emph{Gaussian curvature} of $\Sigma$ is $\det B$. A surface with constant 
Gaussian curvature equal to $K$ is called \emph{CGC}-$K$.
\end{defi}

By Gauss' equation \eqref{eq gauss}, $\Sigma$ is a CGC-$K$ surface if and only if the first fundamental form has constant intrinsic curvature $-K$.

\begin{example} \label{example hyperboloid} (See Figure \ref{fig:hyperboloid})
The future sheet of the two-sheeted hyperboloid
$$\mathrm{Hyp}:=\{\pmb x\in\R^{2,1}\,:\,\langle \pmb x,\pmb x\rangle=-1, x_3 > 0\}~,$$
is CGC-1. Since it is simply connected and the first fundamental form $\I$ is a complete hyperbolic metric (i.e. of constant intrinsic curvature $-1$), $\mathrm{Hyp}$ is isometric to the hyperbolic plane $\Hyp^2$. The normal vector of $\mathrm{Hyp}$ at a point $\pmb p$ is $\pmb n(\pmb p)=\pmb p$, hence the shape operator of $\mathrm{Hyp}$ is the identity. When considered as a surface in its own right we will use the notation $\mathrm{Hyp}$ and when viewed as the target of the Gauss map (see below) of any surface, we will refer to it as $\Hyp^2$.
\end{example}
\begin{figure}[htb]
\centering
\includegraphics[height=4.5cm]{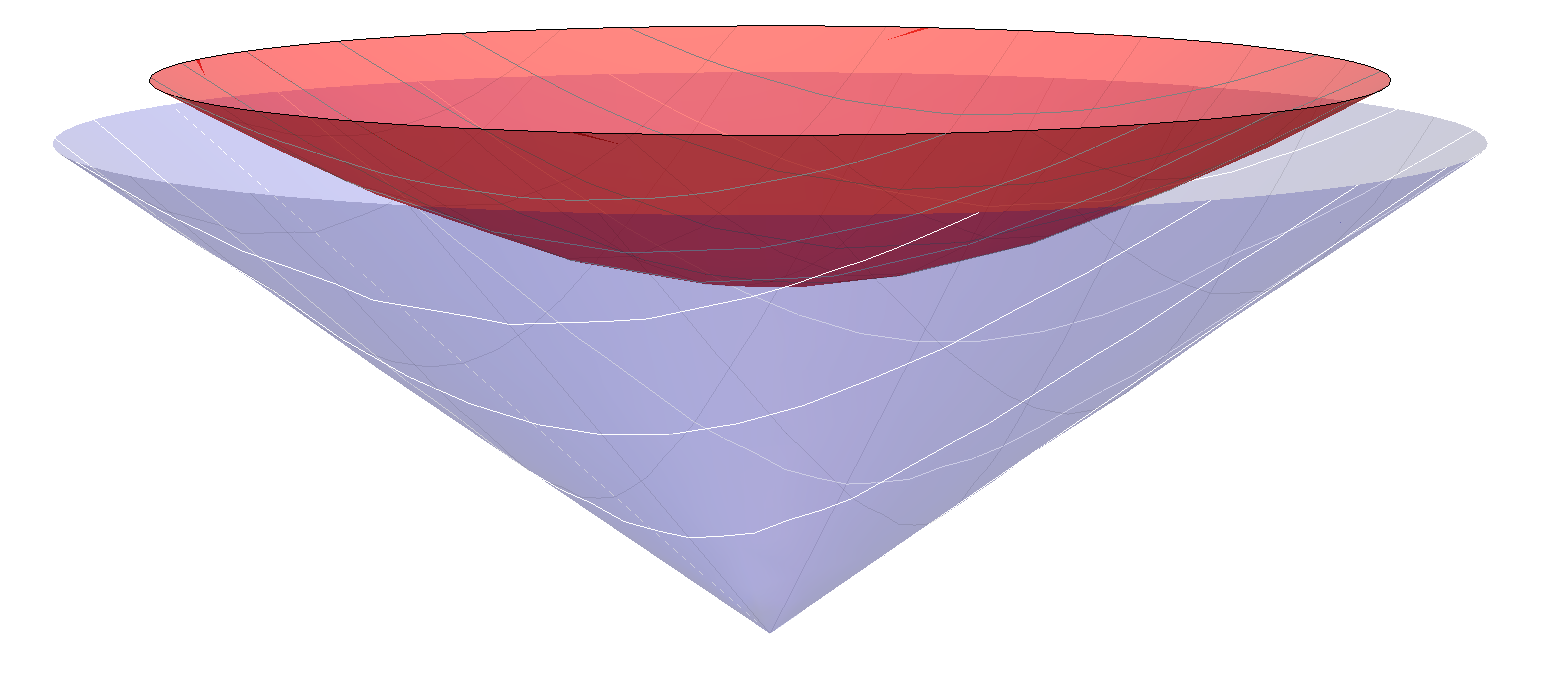}
\caption{The hyperboloid $\mathrm{Hyp}$, whose domain of dependence is the cone $I^+(\pmb 0)$. \label{fig:hyperboloid}}
\end{figure}

\begin{example} \label{ex support function flat}
Define the trough $T$ by
$$T:=\{\pmb x\in\R^{2,1}\,:\,x_2^2-x_3^2=-1,\,x_3>0\}~.$$
It can be described as the cartesian product of a hyperbola $x_2^2-x_3^2=-1$ and a line. The eigenvalues of the shape operator of $T$ are $1$ and 0, so it has zero Gaussian curvature. See Figure \ref{fig:ruled}.
\end{example}

\begin{figure}[htb]
\centering
\includegraphics[height=4.5cm]{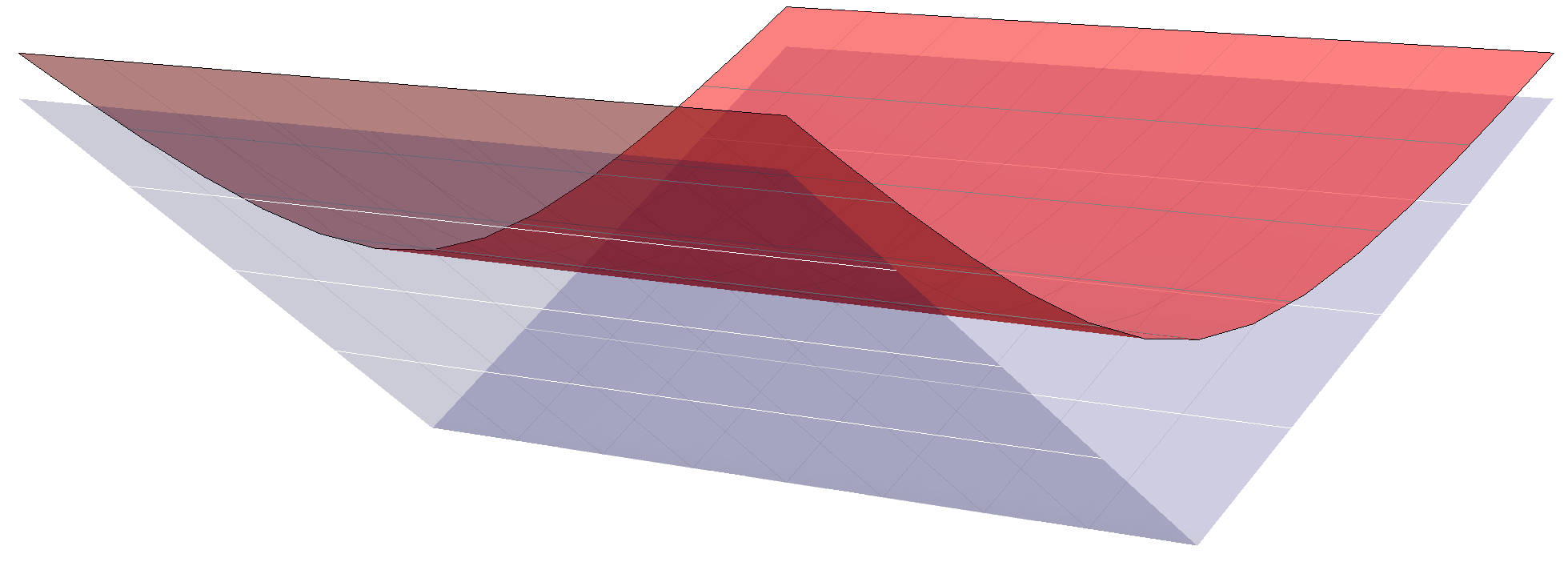}
\caption{The trough $T$, whose domain of dependence is the wedge $I^+(\ell)$. \label{fig:ruled}}
\end{figure}

The \emph{Gauss map} of a $C^1$ spacelike surface $\Sigma$, analogously to the Euclidean case, is the function
$$G_\Sigma:\Sigma\to\Hyp^2~,$$
defined by 
$$G_\Sigma(\pmb p)=\pmb n(\pmb p)~,$$
where $\pmb n$ is the future unit normal vector of $\Sigma$, considered as a point of $\Hyp^2$. Since the shape operator $B$ is the derivative of the Gauss map, the third fundamental form is the pull back under $G_\Sigma$ of the hyperbolic metric $h_{\Hyp^2}$ on $\Hyp^2$. 

\subsection{CGC surfaces and minimal Lagrangian maps}
Let us now explain the relation between surfaces of constant Gaussian curvature and   minimal Lagrangian diffeomorphisms between hyperbolic surfaces.

\begin{defi} \label{defi minimal lagrangian}
Given two hyperbolic surfaces $(S,h)$ and $(S',h')$, a diffeomorphism $F:(S,h)\to (S',h')$ is \emph{minimal Lagrangian} if the unique positive definite $h$-symmetric tensor $b\in\Gamma(\mathrm{End}(TS))$ such that $F^*h'=h(b\cdot,b\cdot)$ satisfies the Codazzi equation $d^{\nabla_h} b=0$, where $\nabla_h$ is the Levi-Civita connection of $h$.
\end{defi}

\begin{remark} \label{remark det = 1}
The tensor $b$ can also be described as the symmetric part of the polar decomposition of the linear map $dF$ with respect to the inner products $h$ and $h'$. If $F$ is a minimal Lagrangian map between hyperbolic surfaces, it follows that $\det b = 1$ \cite{labourieCP}.
\end{remark}

\begin{lemma} \label{lemma min lag}
Given any convex CGC-$K$ surface $\Sigma^K$ in $\R^{2,1}$, with first fundamental form $\I$, the Gauss map of $\Sigma^K$ is a minimal Lagrangian map, when considered as a map:
$$G:\left(\Sigma^K,K\cdot \I\right)\to(\Hyp^2,h_{\Hyp^2})~.$$
\end{lemma}
\begin{proof}
First of all, observe that, by Gauss' equation \eqref{eq gauss} the intrinsic curvature of $\I$ equals $-K$, and therefore the metric $K\cdot \I$ is a hyperbolic metric. Now, let us take $b=(1/\sqrt K)B$. 
The pull-back of the hyperbolic metric of $\Hyp^2$ by the Gauss map is:
\begin{equation} \label{eq pullback Gauss}
G^*h_{\Hyp^2}(v,w)=\III(v,w)=\I(B(v),B(w)) = K \I(b(v),b(w))~,
\end{equation}
where $B= D\pmb n$ is the shape operator of $\Sigma$. Moreover, since $B$ is self-adjoint and Codazzi for $\I$, then it is also self-adjoint and Codazzi for $K\cdot\I$, and so is $b$.
\end{proof}

\begin{lemma} \label{rmk min lag realizable}
Given a simply connected hyperbolic surface $(S,h)$, possibly not complete, and a minimal Lagrangian local diffeomorphism $F: (S,h) \to \Hyp^2$, there exists an isometric immersion $\sigma:(S,(1/K)\cdot h)\to \R^{2,1}$ with Gauss map equal to $F$.
\end{lemma}
\begin{proof}
Let $b$ be as in Definition \ref{defi minimal lagrangian}. Then, the proof of Lemma \ref{lemma min lag} suggests the ansatz $((1/K)\cdot h,\sqrt K b)$ for the embedding data of a CGC-$K$ surface. It then follows from Remark \ref{remark det = 1} that the pair $((1/K)\cdot h,\sqrt K b)$ satisfies the equations of Gauss and Codazzi. Hence by Theorem \ref{fund theorem}, there exists an immersion $\sigma$ having $((1/K)\cdot h,\sqrt K b)$ as embedding data. 

Moreover, from the definition of $b$ we have $F^*h_{\Hyp^2}=h(b\cdot,b\cdot)$, while from the same computation as in the proof of Lemma \ref{lemma min lag}, $G_\sigma^*h_{\Hyp^2}=(1/K)\cdot h(\sqrt K b\cdot,\sqrt K b\cdot)=h(b\cdot,b\cdot)$. Hence at each point $\pmb p \in S$, $F$ and $G$ differ by an isometry of $\Hyp^2$ in a neighborhood of $\pmb p$. Since $S$ is connected, this isometry must in fact be constant. By postcomposing $\sigma$ with the corresponding isometry of $\R^{2,1}$, we may take it to be the identity.
\end{proof}





\subsection{Entire spacelike surfaces}

In this paper, we will study \emph{entire} embedded spacelike surfaces. Let us introduce this notion.

\begin{defi}
An achronal surface in $\R^{2,1}$ is entire if  $\pi|_{\Sigma}:\Sigma\to\R^2$ is a homeomorphism, where $\pi:\R^{2,1}\to\R^2$ is the vertical projection $\pi(x_1,x_2,x_3)=(x_1,x_2)$.
\end{defi}

Entire achronal surfaces are exactly the graphs of 1-Lipschitz functions on $\R^2$. Entire spacelike surfaces are exactly the graphs of $C^1$ and strictly 1-Lipschitz functions on $\R^2$. Clearly an entire surface is properly immersed. The following elementary proposition says that the converse is true as well.

\begin{prop} \label{prop properly emb iff entire}
Every properly immersed achronal surface in $\R^{2,1}$ is entire.
\end{prop}

\begin{proof}

Let $\Sigma$ be a properly immersed achronal surface. By the achronal condition, the projection $\pi: \Sigma \to \R^2$ is a local homeomorphism. We now prove that $\pi$ has the path lifting property: given a point $\pmb p$ is $\Sigma$ and a curve $\gamma: [0,1] \to \R^2$ with $\gamma(0) = \pi(\pmb p)$, there is a lift $\tilde{\gamma}: [0,1] \to \Sigma$ with $\pi \circ \tilde{\gamma} = \gamma$. Let $\gamma:[0,1] \to \R^2$ be such a curve. Since $\pi$ is a local homeomorphism, the path $\gamma$ can be lifted to an open neighborhood. Since $\Sigma$ is achronal, the length of any partial lift $\tilde{\gamma}$ measured using the Euclidean metric on $\R^3$ is at most $\sqrt{2}$ times the length of $\gamma$ in $\R^2$. Since the immersion is proper, the induced \emph{Euclidean} metric on $\Sigma$ is complete. As a consequence, the partial lift of $\gamma$ can also be extended to all limit points. Therefore the interval on which we can lift $\gamma$ is both open and closed, so it is the entire interval $[0,1]$. 

We have shown that $\pi$ is a local homeomorphism with the path lifting property, so it is a covering map \cite[p. 383]{Carmo:1976aa}. But the image $\R^2$ is simply connected, so $\pi$ must be a homeomorphism.
\end{proof}


\begin{remark}
Proposition \ref{prop properly emb iff entire} shows that the condition of being  entire is preserved by isometries of $\R^{2,1}$. In other words, if $\pi|_{\Sigma}$ is a homeomorphism, then the orthogonal projection from $\Sigma$ to \emph{any} spacelike plane is a homeomorphism.
\end{remark}

\begin{remark} \label{rmk complete implies entire}
The projection $\pi$ is distance non-decreasing. Therefore, if the first fundamental form of a spacelike surface $\Sigma$ is a complete Riemannian metric, then $\Sigma$ is necessarily entire. The converse is false; a counterexample will be provided by the entire surface studied in Section \ref{sec triangular surfaces}, whose fundamental form is isometric to an ideal triangle in $\Hyp^2$. See also \cite[Appendix A]{Bonsante:2015vi} for another counterexample.
\end{remark}

\subsection{Domains of dependence}
Recall that a continuous curve $\gamma:I\to\R^{2,1}$ is called \emph{causal} if for all pairs of points $t,s \in I$, the images $\gamma(t)$ and $\gamma(s)$ differ by a lightlike or timelike vector.

\begin{defi} \label{def domain of dependence}
Given a spacelike surface $\Sigma$ in $\R^{2,1}$, the \emph{domain of dependence} $\mathcal D_\Sigma$ of $\Sigma$, 
is the set of all points $\pmb p\in\R^{2,1}$ such that every inextendable causal curve through $\pmb p$ intersects $\Sigma$.
\end{defi}

Let us provide the following description of domains of dependence for entire spacelike surfaces. We say that a half-space is \emph{null} if it is bounded by a lightlike plane. An open null half-space is equal either to the future or to the past of its boundary plane. 

\begin{lemma} \label{lemma domains dependence} 
If $\Sigma$ is an entire spacelike surface in $\R^{2,1}$, then its domain of dependence $\mathcal D_\Sigma$ is open, and is equal to the intersection of the open null half-spaces containing $\Sigma$. Moreover, exactly one of the following holds:
\begin{enumerate}
\item $\mathcal D_\Sigma = \R^{2,1}$;
\item $\mathcal D_\Sigma = I^+(Q) \cap I^-(P)$ where $Q$ and $P$ are parallel null planes, with $P$ lying in the future of $Q$;
\item $\mathcal D_\Sigma=\bigcap_{Q\in\mathcal F} I^+(Q)$ where $\mathcal F$ is a nonempty family of null planes; or
\item $\mathcal D_\Sigma=\bigcap_{Q\in\mathcal F} I^-(Q)$ where $\mathcal F$ is a nonempty family of null planes.
\end{enumerate}
\end{lemma}

\begin{proof}[Proof of Lemma \ref{lemma domains dependence}]
We divide the proof into several steps.
\begin{steps}
\item \emph{We first prove that $\mathcal{D}_\Sigma$ is open.} Let $\mathrm{Caus}_{\R^{2,1}}$ be the space of all inextendable causal curves in $\R^{2,1}$, with the topology of local uniform convergence. Since every such curve is the graph of a 1-Lipschitz function from $\R$ to $\R^2$, the Arzel\`a-Ascoli theorem implies that for any compact set $K \in \R^{2,1}$, the subset $\mathrm{Caus}_K$ of such curves intersecting $K$ is compact. We now show that the set $\mathrm{Caus}_\Sigma$ of such curves intersecting $\Sigma$ is open. Suppose $\gamma \in \mathrm{Caus}_{\R^{2,1}}$ intersects $\Sigma$ at a point $\pmb p$. Since $\Sigma$ is spacelike, a small circle in $\Sigma$ around $\pmb p$ must be at least some fixed Euclidean distance from the light cone of $\pmb p$. Perturbing $\gamma$ by less than this distance, it must still pass through the circle and hence intersect $\Sigma$.

To complete the proof that $\mathcal{D}_\Sigma$ is open, for any $\pmb p \in \mathcal{D}_\Sigma$, let $K_n$ be a sequence of compact neighborhoods of $\pmb p$ in $\R^{2,1}$ whose intersection is $\pmb p$. Then $\mathrm{Caus}_{\pmb p} = \bigcap_n \mathrm{Caus}_{K_n} \subset \mathrm{Caus}_\Sigma$. Since $\mathrm{Caus}_{K_n}$ are compact and $\mathrm{Caus}_\Sigma$ is open it follows that for $n$ sufficiently large, $\mathrm{Caus}_{K_n} \subset \mathrm{Caus}_\Sigma$ whence $K_n \subset \mathcal{D}_\Sigma$.

\item \emph{We show that every open null half-space containing $\Sigma$ also contains $\mathcal{D}_\Sigma$.} Let $H$ be an open null half-space containing $\Sigma$. For any point $\pmb p\notin H$, the null line through $\pmb p$ parallel to the boundary of $H$ lies entirely outside of $H$. Since $\Sigma$ is contained in $H$, this line exhibits a causal curve containing $\pmb p$ which does not meet $\Sigma$, showing that $\pmb p \notin \mathcal{D}_\Sigma$. Therefore $\mathcal{D}_\Sigma \subset H$.

\item
\emph{We prove that $\mathcal{D}_\Sigma$ is the intersection of the open null half-spaces containing $\Sigma$}. By Step 2, $\mathcal{D}_\Sigma$ is contained in this intersection. Now we simply need to show that if $\pmb p \notin \mathcal{D}_\Sigma$, we can find a closed null half space containing $\pmb p$ but not $\Sigma$. 

Let $\pmb p$ be a point not in $\mathcal{D}_\Sigma$. Since $\Sigma \subset \mathcal{D}_\Sigma$, $\pmb p$ is either in the past or the future of $\Sigma$. If $\pmb p$ is in the past of $\Sigma$, then any point in $\overline{I^-(\pmb p)}$ can be connected to $\pmb p$ by a causal geodesic which does not meet $\Sigma$. Hence if one can ``escape'' $\Sigma$ from $\pmb p$, one can also escape $\Sigma$ from any point in the past of $\pmb p$, so all of $\overline{I^-(\pmb p)}$ must be outside of $\mathcal{D}_\Sigma$. Similarly, if $\pmb p \in I^+(\Sigma)$, then all of $\overline{I^+(\pmb p)}$ must lie outside of $\mathcal{D}_\Sigma$. 

Up to time reversal, we may assume that $\pmb p \in I^-(\Sigma)$. Let $\pmb q$ be a point in $\overline{I^+(\pmb p)}$ which is still below $\Sigma$ but is contained in the boundary of $\mathcal{D}_\Sigma$ (it may be that $\pmb q = \pmb p$). Since $\pmb q \notin \mathcal{D}_\Sigma$, there is an inextendable causal curve $\gamma$ containing $\pmb q$ which does not intersect $\Sigma$. We first show that the part $\gamma^+$ of $\gamma$ in the closed future of $\pmb q$ must be a null geodesic ray. Otherwise, it would contain a point $\pmb r$ which was timelike separated from $\pmb q$, and so by the previous paragraph, $\overline{I^-(\pmb r)}$ would be disjoint from $\mathcal{D}_\Sigma$. But $I^-(\pmb r)$ contains an open neighborhood of $\pmb q$, which contradicts $\pmb q \in \partial \mathcal{D}_\Sigma$.

Let $H = I^-(\gamma^+)$. This is the unique open past-complete null half-space containing $\gamma^+$ in its boundary. By the same reasoning as above, $H$ cannot intersect $\mathcal{D}_\Sigma$, and since $\mathcal{D}_\Sigma$ is open, neither can $\overline{H}$. But $\pmb q$ and $\pmb p$ are both in $\overline{H}$, which completes the proof.

\item
\emph{We prove that exactly one of the four options must hold.} It is enough to observe that if $\mathcal{D}_\Sigma$ is contained in the intersection of a past-complete null half-space $H^-$ and a future-complete null half-space $H^+$ then the boundaries of $H^+$ and $H^-$ must be parallel. Otherwise, the projection of $\mathcal{D}_\Sigma$ to $\R^2$ could not be surjective, but it must be since $\Sigma \subset \mathcal{D}_\Sigma$ and $\Sigma$ is entire. \qedhere
\end{steps}
\end{proof}

We have the following definition of future-complete domains and future-convex spacelike surfaces.
\begin{defi} An entire achronal surface $\Sigma$ is called \emph{future-convex} (resp. \emph{strictly future-convex}) if $I^+(\Sigma)$ is future-complete and convex (resp. strictly convex).
\end{defi}

\begin{remark}
The condition that a $C^2$ entire spacelike surface $\Sigma$ is future-convex is equivalent to the fact that the shape operator $B= D\pmb n$ (where $\pmb n$ is the \emph{future} unit normal vector field) is positive semi-definite. Hence these are surfaces having non-negative mean curvature and Gaussian curvature, namely $\tr B\geq 0$ and $\det B\geq 0$. 
\end{remark}

From Lemma \ref{lemma domains dependence}, we therefore have the following characterization of domains of dependence of future-convex entire surfaces:

\begin{cor} \label{cor domains dependence}
If $\Sigma$ is a future-convex entire spacelike surface in $\R^{2,1}$, then
 $\mathcal D_\Sigma$ is a convex open domain of the form
$$\mathcal D_\Sigma=\bigcap_{Q\in\mathcal F} I^+(Q)\qquad~,$$
where $\mathcal F$ is a (possibly empty) family of null planes. We can take $\mathcal{F}$ to be the family of all null planes containing $\Sigma$ in their future. 
\end{cor}

\begin{remark}
There is clearly an analogous definition of past-complete domains and past-convex surfaces. Any isometry of $\R^{2,1}$ which is \emph{not} future-preserving exchanges future-complete  domains with past-complete  domains, and future-convex surfaces with past-convex surfaces. For this reason, we will always assume without loss of generality that our surfaces are future-convex with future-complete domains of dependence.
\end{remark}

\begin{example} \label{example hyperboloid2} 
The domain of dependence of the hyperboloid $\mathrm{Hyp}$ of Example \ref{example hyperboloid} is the future cone over the origin, namely:
$$\mathcal D_{\Hyp^2}=I^+(\pmb 0)~.$$
This is the intersection of all the future half-spaces bounded by a null plane through the origin.
\end{example}

\begin{example}\label{rmk wedge}
 The domain of dependence of the trough $T$ of Example \ref{ex support function flat} is the \emph{wedge}:
$$W=I^+(Q^1)\cap I^+(Q^2)~,$$
where $Q^1$ and $Q^2$ are two non-parallel planes, which intersect along a spacelike line $\ell$.
 Namely, $\partial W$ is composed of two null half-planes, both having the same spacelike line $\ell$ as a boundary.  See again Figure \ref{fig:ruled}.
\end{example}


Let us observe that, if the family $\mathcal F$ of Corollary \ref{cor domains dependence} is empty, then $\mathcal D_\Sigma=\R^{2,1}$, while if  $\mathcal F$ contains only one element $Q$, then  $\mathcal D_\Sigma$ is the future of the null plane $Q$.
In Example \ref{rmk wedge}, we can assume $\mathcal F$ is composed of exactly two non-parallel null planes.
We will say that  $\mathcal D_\Sigma$ is a (future-complete) \emph{regular domain} if $\mathcal F$ contains at least two {non-parallel} elements. More precisely:

\begin{defi}
A convex open domain $\mathcal D\subset \R^{2,1}$ is a \emph{regular domain} if
$$\mathcal D=\bigcap_{Q\in\mathcal F} I^+(Q)~,$$
for some family $\mathcal F$ of null planes which contains at least two non-parallel distinct planes. 
\end{defi}

\subsection{Minkowski problem in regular domains} \label{subsec mink problem}

With these preliminary remarks in hand, we can formulate more precisely the statement of the problems we consider in this paper.
Let us denote by $\kappa_\Sigma:\Sigma\to\R$ the Gaussian curvature of a spacelike surface $\Sigma$. The Minkowski problem we consider can be stated as follows:

\begin{namedthm}{Minkowski problem}
Given any regular domain $\mathcal D$ in $\R^{2,1}$ and any sufficiently regular function $\psi:\Hyp^2\to\R^{>0}$, does there exist a unique entire surface $\Sigma$  such that
\begin{enumerate}
\item $\psi\circ G_\Sigma=\kappa_\Sigma,$ and
\item $\mathcal D_{\Sigma} = \mathcal D~$?
\end{enumerate}
\end{namedthm}

We will give a positive answer (Theorem \ref{thm mink problem}) to the Minkowski problem, under the assumption that $\mathcal D$ is not a wedge (compare Example \ref{rmk wedge}), which we will show is also a necessary condition.

\begin{remark}
Let us make some remarks on the formulation of the problem.
\begin{enumerate}
\item Consistently with the classical Minkowski problem in Euclidean space, we will consider the Minkowski problem for a prescribed \emph{positive} function $\psi$ on $\Hyp^2$. This implies that a surface $\Sigma$ is strictly convex --- that is, either $I^+(\Sigma)$ or $I^-(\Sigma)$ is a strictly convex domain (with smooth boundary equal to $\Sigma$). 
\item We will give an affirmative answer to the Minkowski problem --- both for the existence and uniqueness part --- under the assumption that 
$$a<\psi<b$$
for some constants $a,b>0$. Without such assumption, the problem appears significantly more complicated, at least with the tools of this paper and of the existing literature.
\item We shall prove in Section \ref{sec image of gauss map and support function on boundary}
 that, for every entire spacelike surface $\Sigma$ with Gaussian curvature bounded from above and below by positive constants (as in the previous point), the image of the Gauss map $G_\Sigma$ coincides with the image of the subdifferential of $\partial\mathcal D_\Sigma$ --- that is, with the set of vectors $\pmb v\in\Hyp^2$ such that $\mathcal D_\Sigma$ admits a support plane orthogonal to $\pmb v$. Hence the function $\psi$ need only be defined on the image of the Gauss map of $\partial\mathcal D_\Sigma$.

\item If $\Sigma$ is a strictly convex smooth entire surface in $\R^{2,1}$, then its Gauss map $G_\Sigma:\Sigma\to\Hyp^2$ is a diffemorphism onto its image.  Hence under our assumptions, the condition of the Minkowski problem can also written as
$$\psi=\kappa_\Sigma\circ G_\Sigma^{-1}~.$$
The function $\kappa_\Sigma\circ G_\Sigma^{-1}$ is also called \emph{curvature function} of $\Sigma$.
\end{enumerate}
\end{remark}

A particular case is obtained when the prescribed curvature function is constant. In Theorem \ref{thm-CGC-problem} we will give a positive answer, under the necessary and sufficient condition that the regular domain $\mathcal D$ is not a wedge, to the following problem:

\begin{namedthm}{CGC problem}
Given any regular domain $\mathcal D$, 
does there exist for every $K>0$ a unique entire CGC-$K$ surface $\Sigma^K$ such that its domain of dependence $\mathcal D_{\Sigma^K}$ is the prescribed regular domain $\mathcal D$?
\end{namedthm}

\section{Analytical formulation} \label{analytical sec}

The purpose of this section is to translate the study of convex surfaces in Minkowski space in analytical terms, with particular focus on the aforementioned Minkowski problem. That is, we introduce the \emph{support function} for convex spacelike surfaces and we express the Minkowski problem in terms of a partial differential equation of Monge-Amp\`{e}re type.

\subsection{Support functions}

It will be convenient to introduce the following definitions from the theory of convex functions:

\begin{defi} \cite{rockafellar} \label{def closed covex}
A function $\R^n \to \R \cup \{\pm \infty\}$ is called \emph{convex}, resp. \emph{closed}, if its supergraph $\{(\mathsf x, z) \subset \R^n \times \R\, |\, f(\mathsf x) < + \infty \textrm{ and } z \geq f(\mathsf x)\}$ is convex, resp. closed. A function $f$ is \emph{proper} if $f(\mathsf x) < + \infty$ for at least one $\mathsf x$ and $f(\mathsf x) > -\infty$ for every $\mathsf x$. The \emph{essential support} of a convex function $f$ is the set on which $f$ is finite.
\end{defi}

Except for minor technicalities, we are concerned only with proper functions. However, it is essential that we consider functions which take the value $+ \infty$ at some points, so we will henceforth allow all our functions to be infinite without further ado. Note that a function is closed if and only if it is lower semi-continuous. If $X$ is a subset of $\R^n$, we will say $f$ is a function on $X$ if it is a (proper) function on $\R^n$ with essential support contained in $X$.

In the following definition, we are interested especially in the case where the set $S$ is a future-convex entire spacelike surface and the case where $S$ is a domain of dependence.

\begin{defi} \label{defi support functions} 
Let $S$ be a nonempty subset of $\R^{2,1}$. Then
the \emph{support function} of S is the function $u_S:\overline \D \to\R\cup\{+\infty\}$
defined by 
$$u_S(\mathsf{y})=\sup_{\pmb x\in S}\langle \pmb  x,(\mathsf{y},1)\rangle~,$$
where $\overline\D$ is the closed unit disk in $\R^2$ and $\mathsf{y}=(y_1,y_2)\in \overline\D$.
\end{defi}

Observe that the plane $\{ \pmb x \in \R^{2,1}\, |\, \langle \pmb x, (\mathsf y, 1) \rangle = z\}$ is spacelike for $\mathsf y$ in the interior of the unit disk and null for $\mathsf y$ on the unit circle. In fact, as $\mathsf y$ and $z$ range over $\overline\D \times \R$, this parametrizes all spacelike and null planes in $\R^{2,1}$. Adorned with an appropriate geometric structure, the space of such planes is known in the literature as co-Minkowski space \cite{surveyseppifillastre} or half-pipe geometry \cite{Danciger:2013ab}. Of course, we could just as well think of it as the space of all future-complete half-spaces in $\R^{2,1}$ with spacelike or null boundary.

For our purposes, we are concerned only with the topology and convexity of this space. Recall that if $f$ is a function on $\R^2$ valued in $\R \cup \{+ \infty\}$ and not identically equal to $+\infty$, the Legendre transform of $f$ is the function $f^*: \R^2 \to \R \cup \{+ \infty\}$ defined by
$$f^*(\mathsf{y})=\sup_{\mathsf{x}\in\R^2}(\mathsf{y} \cdot \mathsf{x} -f(\mathsf{x}))~.$$
It follows from the definitions of support function and Legendre transform that, if an achronal $\Sigma$ is the graph of some function $f:\R^2\to\R$, then its support function $u_\Sigma$ equals the Legendre transform $f^*$ restricted to $\overline\D$. Moreover, $f^*$ is $+\infty$ outside $\overline\D$.

\begin{prop}\cite[Cor 12.2.1]{rockafellar} The Legendre transform gives an involutive one-to-one correspondence between proper closed convex functions on $\R^2$.
\end{prop}

Restricting to functions on the closed disk and associating an entire achronal surface (which is the graph of a 1-Lipschitz function) with the function of which it is a graph, an immediate corollary is the following version of convex duality:

\begin{prop} \label{prop: convex duality}
The Legendre transform gives an involutive bijection between entire convex achronal surfaces and proper closed convex functions on $\overline{\D}$.
\end{prop}

Now we concentrate on the support function of a regular domain. Recall that a regular domain $\mathcal D$ is an open domain which can be written as the intersection of the futures of a family $\mathcal{F}$ of at least two nonparallel null planes in $\R^{2,1}$. Thinking of $\D \times \R$ as the space of null or spacelike planes in $\R^{2,1}$, we view $\mathcal{F}$ as a subset of $\partial\D \times \R$. The family $\mathcal{F}$ is not unique -- for instance, we may add to the family a null plane parallel to and lying below a plane already in $\mathcal{F}$ without changing the domain $\mathcal D$. However, the union of defining families is still a defining family, so given a regular domain $\mathcal D$ we may consider the maximal family $\mathcal{F}_{\mathcal D}$ of defining planes. Since $\mathcal D$ is assumed to be open, a limit of planes disjoint from $\mathcal D$ is still disjoint from $\mathcal D$, so since $\mathcal{F}_{\mathcal D}$ is maximal it must be closed as a subset of $\partial\D \times \R$. Since it is also upward-closed, $\mathcal{F}_{\mathcal D}$ is the supergraph of a closed function $\ph_{\mathcal D}$ on $\partial\D$. Note that $\ph_{\mathcal D}$ is finite at at least two points because the set $\mathcal{F}_{\mathcal D}$ by assumption contains at least two non-parallel planes. As a consequence we obtain the following proposition:


\begin{prop} \label{prop:bijection reg dom} The assignment $\mathcal D \mapsto \ph_{\mathcal D}$ is a bijection between the set of regular domains and the set of proper closed functions on the circle which are finite at at least two points.
\end{prop}
We will use the notation $\mathcal D_\ph$ to represent the domain corresponding to $\ph$.

Another important notion of convex geometry is the convex envelope.


\begin{defi}\label{defi convex envelope}
If $f$ is any function on $\R^2$ valued in $\R \cup \{+ \infty\}$, the convex envelope $\mathrm{conv}(f)$ is the function whose supergraph is the closure of the convex hull of the supergraph of $f$. 
\end{defi}

Equivalently \cite[Cor 12.1.1]{rockafellar}, $\mathrm{conv} (f)$ can be equivalently expressed as the supremum of affine functions less than or equal to $f$:
$$\mathrm{conv} (f) (\mathsf x)=\sup\{v(\mathsf x)\,|\,v:\R^2\to\R\text{ is affine, }v\leq f\}~.$$

\begin{prop} \label{prop reg domain conv} Let $\mathcal D$ be a regular domain. Then the support function $u_{\mathcal D}$ is equal to $\mathrm{conv}(\ph_{\mathcal D})$. Moreover $u_{\mathcal D}$ restricted to the unit circle is equal to $\ph_{\mathcal D}$ and if $\ph_{\mathcal D}$ is infinite on an open arc with endpoints $\xi_1$ and $\xi_2$, then $u_{\mathcal D}$ restricted to the chord $[\xi_1, \xi_2]$ is the convex envelope of $\ph_{\mathcal D}|_{\{\xi_1,\xi_2\}}$.  
\end{prop}

Let us write $\ph = \ph_{\mathcal D}$. The last property of $u_{\mathcal D}$ says that $u_{\mathcal D}$ restricted to the open chord $(\xi_1,\xi_2)$ is infinite if either $\ph(\xi_1)$ or $\ph(\xi_2)$ are infinite, and otherwise is the unique affine function interpolating $\ph(\xi_1)$ and $\ph(\xi_2)$. Note that this also implies that the essential support of $\mathrm{conv}(\ph)$ is the convex hull of the essential support of $\ph$.

\begin{proof}
By construction, $\mathcal D$ is the strict supergraph of the Legendre transform $\ph^*$. Since the support function of $\mathcal D$ is the same as the support function of its closure and the support function is the restriction of the Legendre transform to the disk, $u_{\mathcal D} = \ph^{**}$. By \cite[Thm 12.2]{rockafellar}, $\ph^{**} = \mathrm{conv}(\ph)$. 

We now show that as long as $\ph$ is lower semi-continuous, $\mathrm{conv}(\ph)$ restricted to the unit circle is equal to $\ph$. Let $\ph^+$ be the supergraph of $\ph$. By assumption it is closed, and the first thing we need to show is that its convex hull is still closed. According to \cite[Cor 17.2]{rockafellar}, if $S$ is a bounded set of points in $\R^n$, then $\mathrm{cl}(\mathrm{conv}(S)) = \mathrm{conv}(\mathrm{cl}(S))$. We would like to apply this theorem with $S = \ph^+$, but since it is not bounded so we need a slightly generalized theorem. If we include $\R^2 \times \R$ into $\mathbb{RP}^3$, then the union of $\ph^+$ with the point at $z = + \infty$ is still closed, and after a projective transformation it is bounded in $\R^2 \times \R$. Applying the closure theorem to this transformed set and then transforming back, we conclude that the convex hull of $\ph^+$ is closed.

Therefore, the supergraph of $\mathrm{conv}(\ph)$ is actually the convex hull of $\ph^+$, not just its closure. Hence any point in the supergraph of $u_{\mathcal D}$ is a convex linear combination of finitely many points in $\ph^+$. If $\ph$ is supported on only one side of a line $L$, then each point in the graph of $u_{\mathcal D}|_L$ is a convex linear combination of only those points in $\ph^+|_L$. 

Applying this observation to the case where $L$ is tangent to the unit circle, we see that $u_{\mathcal D}$ restricted to the unit circle is equal to $\ph$. Applying the observation to the case where $L$ contains a chord $[\xi_1,\xi_2]$ as in the statement of the proposition, we conclude that $u_{\mathcal D}$ restricted to $[\xi_1,\xi_2]$ is the convex envelope of $\ph|_{\{\xi_1,\xi_2\}}$. 
\end{proof}

\subsection{A Dirichlet-type problem}

In this section, we characterize the support function of an entire future-convex spacelike surface with prescribed Gaussian curvature, and show that the problem of finding such a surface in a given domain of dependence is dual to a Dirichlet-like problem for the support function. 

In the following, let $\Sigma$ be an entire future-convex spacelike surface in $\R^{2,1}$. By Corollary \ref{cor domains dependence}, the domain of dependence of $\Sigma$ is the intersection of the future-complete open null half-spaces containing it. Hence we have the following lemma:


\begin{lemma} Let $\Sigma$ be an entire spacelike future-convex surface with domain of dependence $\mathcal D$. Let $u_\Sigma$ be the support function of $\Sigma$ and let $u_{\mathcal D}$ be the support function of $\mathcal D$. Then $u_\Sigma$ and $u_{\mathcal D}$ coincide on the unit circle.
\end{lemma}

The fact that $\Sigma$ is entire gives another restriction on its support function $u_\Sigma$. In order to describe this condition, we first define the \emph{domain of support} of a proper closed convex function to be the interior of its essential domain, i.e. the largest open set on which $u$ is finite. We will use $\Omega_u$ to denote the domain of support of $u$. Since $u$ is convex, so is its essential domain, which implies that the essential domain of $u$ is either contained in a line or has nonempty interior. Setting the first possibility aside for the moment, assume that $\Omega_u$ is nonempty. The essential domain of $u$ is contained in $\overline{\Omega}_u$ and since the function $u$ is closed its values on the boundary of $\Omega_u$ are uniquely determined by its restriction to $\Omega_u$.

Now we may make the definition:
\begin{defi} \label{defi gradient surjective}
Let $u$ be a proper closed convex function on $\R^2$ such that $\Omega_u$ is nonempty and bounded and $u$ is differentiable throughout $\Omega_u$. The function $u$ is called \emph{gradient surjective} if its gradient map $Du: \Omega_u \to \R^2$ is surjective.

\end{defi}

By a special case of \cite[Thm 26.3]{rockafellar}, a  function $u$ is gradient surjective if and only if its Legendre transform $u^*$ is entire and strictly convex. By convex duality (Proposition \ref{prop: convex duality}), this implies:



\begin{lemma}
The support function of a strictly future-convex entire spacelike surface is gradient surjective.
\end{lemma}

Applying a variant of the same theorem \cite[Thm 26.3]{rockafellar} to $u_\Sigma^*$, we also see that if $\Sigma$ is $C^1$ as well as being strictly convex, then the gradient map $Du_\Sigma$ is injective as well, so by invariance of domain it gives a homeomorphism from $\Omega_u$ to $\R^2$. We remark that this gradient is related to the inverse of the Gauss map of $\Sigma$. Namely, let us denote by $\pi:\Hyp^2\to\D$ the radial projection from the hyperboloid to the disc at height one, namely
$$\pi(y_1,y_2,y_3)= \left(\frac{y_1}{y_3},\frac{y_2}{y_3}\right)~,$$
which gives an identification of the hyperboloid model of Example \ref{example hyperboloid} with the Klein model of the hyperbolic plane. Then the composition $\pi \circ G_\Sigma$ of the projection with the Gauss map is inverse to the map 
\begin{equation} \label{eq inverse support function}
\mathsf y \mapsto (Du_\Sigma(\mathsf y), u_\Sigma^*(Du_\Sigma(\mathsf y)))
\end{equation}
as maps between $\Sigma$ and $\Omega_{u_\Sigma}$ \cite[Lemma 2.15]{bonfill}. Moreover, we note that if $\Sigma$ is convex then
\begin{equation} \label{eq support function of convex surface}
u_\Sigma(\pi \circ G_\Sigma(\pmb p)) = \langle \pmb p, (\pi \circ G_\Sigma(\pmb p),1) \rangle
\end{equation}
and if $\Sigma$ is entire then $u_\Sigma(\mathsf y) = + \infty$ if $\mathsf y \notin \overline{\pi \circ G_\Sigma(\Sigma)}$.

We now provide a formula which relates the Gaussian curvature of a $C^2$ strictly convex spacelike surface $\Sigma$ to the support function $u_\Sigma$.

\begin{lemma} \cite{Li}\label{lemma gaussian curvature monge}
Let $u_\Sigma:\D\to\R$ be the support function of a future-convex $C^2$ spacelike embedded surface $\Sigma$ in $\R^{2,1}$.
Then $u_\Sigma$ satisfies
\begin{equation} \label{eq monge ampere curvature}
\det D^2 u_\Sigma(\mathsf x)=\frac{1}{\psi(\mathsf x)}(1-|\mathsf x|^2)^{-2}~.
\end{equation}
for every $\mathsf x \in\Omega_{u_\Sigma}$, where 
$\psi=\kappa_\Sigma \circ (\pi\circ G_\Sigma)^{-1}$ is the curvature function, and $\kappa_\Sigma =\det B$ is the Gaussian curvature of $\Sigma$. 
\end{lemma}

In particular, if $\Sigma$ is a future-convex surface of constant Gaussian curvature $\det B\equiv K>0$ (as in Definition \ref{defi HKsurface}), then {on the image of the Gauss map} $u_\Sigma$ satisfies:
\begin{equation} \label{eq monge constant curvature}
\det D^2 u_\Sigma(\mathsf x)=\frac{1}{K}(1-|\mathsf x|^2)^{-2}~.
\end{equation}

At last we are ready to translate our original problem of prescribed Gaussian curvature into a Dirichlet-like problem for the support function.


\begin{defi} \label{def entire solution}
Let $\ph:\partial\D\to\R\cup\{+\infty\}$ be lower semicontinuous and let $\psi:\D\to\R$. We say that a proper closed convex function $u: \D\to \R\cup\{+\infty\}$ is a \emph{solution} of the Minkowski problem with curvature function $\psi$ and boundary data $\ph$ if
\begin{itemize}
\item $u$ is equal to $\ph$ when restricted to $\partial\D$,
\item $u \in C^2(\Omega_u)$ and solves the equation $$\det D^2 u(\mathsf x)=\frac{1}{\psi(\mathsf x)}(1-|\mathsf x|^2)^{-2}~,$$
on the domain $\Omega_u$.
\end{itemize}
\end{defi}


With this definition, we obtain an equivalent formulation of the Minkowski problem, as stated in Section \ref{subsec mink problem}:

\begin{prop}
Given any $\ph:\partial\D\to\R\cup\{+\infty\}$ lower semicontinuous and finite at at least 3 points, and any $\psi:\D\to\R$ smooth, $u$ is a gradient-surjective solution of the Minkowski problem with data $\ph$ and $\psi$ if and only if $u$ is the support function of an entire spacelike surface $\Sigma$ such that $\mathcal D_\Sigma=\mathcal D_\ph$ and $\psi=\kappa_\Sigma\circ (\pi\circ G_\Sigma)^{-1}$. 
\end{prop}

\subsection{Gaussian curvature and examples}

Let us now give two first basic explicit examples:

\begin{example} \label{ex support function hyperboloid}
The hyperboloid $\mathrm{Hyp}$ (see Example \ref{example hyperboloid}), rescaled by a factor $1/\sqrt K$, is an entire strictly future-convex surface (which we denote $\mathrm{Hyp}^K$) of constant Gaussian curvature $K$. In fact, it can be checked directly that (if $\pmb n$ is the future unit normal field) its shape operator is $B=D\pmb n=\sqrt K\mathbbm 1$, where $\mathbbm 1$ is the identity operator. Such surface is invariant by the group of linear isometries $\SO_0(2,1)$. Its support function is:
$$u_{\mathrm{Hyp}^K}(\mathsf x)=-\frac{1}{\sqrt K}\sqrt{1-|\mathsf x|^2}~,$$
which is a solution of Equation \eqref{eq monge constant curvature}.
Observe that $u_{\mathrm{Hyp}^K}$ is finite on the whole disk and $u_{\mathrm{Hyp}^K}=0$ on $\partial\D$.
\end{example}

\begin{example} 
We have introduced in Example \ref{ex support function flat} the trough:
$$T:=\{\pmb x\in\R^{2,1}\,:\,x_2^2-x_3^2=-1,\,x_3>0\}~,$$
Its support function, at any point $\mathsf x=(x,y)$, is:
$$u_{T}(x,y)=\begin{cases} 
-\sqrt{1-y^2} & \text{if }x=0\text{ and }y\in[-1,1] \\
+\infty & \text{otherwise}
 \end{cases}~.$$
 We remark that the trough is convex but not strictly convex and has Gaussian curvature 0. The essential support of $u_T$ is a segment.
\end{example}

\section{Tools from Monge-Amp\`ere equations} \label{sec monge ampere}

In order to prove the existence and uniqueness of entire surfaces of prescribed curvature, we will construct solutions of Equation \eqref{eq monge ampere curvature}. For this purpose, we will need several tools from the classical theory of Monge-Amp\`ere equations --- in particular, the notion of generalized solution, the maximum principle, and some results of existence and regularity. The purpose of this section is to collect the necessary tools and prove a generalized maximum principle for Monge-Amp\`ere equations.

\subsection{Generalized solutions}

Given a convex function $u:\Omega\rar\R$ for $\Omega$ a convex domain in $\R^2$, we define the subdifferential of $u$ as the set-valued function $\partial_u$ whose value at a point $\mathsf x\in\Omega$ is:
$$\partial_u(\mathsf x)=\left\{Dv\,|\,v\text{ affine; }graph(v)\text{ is a support plane for }graph(u);\,(\mathsf x,u(\mathsf x))\in graph(v)\right\}\,.$$

In general $\partial_u(\mathsf x)$ is a convex set. If $u$ is differentiable at $\mathsf x$, then $\partial_u(\mathsf x)=\left\{Du(\mathsf x)\right\}$. We define the Monge-Amp\`ere measure on the collection of Borel subsets $\omega$ of $\R^2$:
$$M\!A_u(\omega)=\mathcal{L}(\partial_u(\omega))$$
where $\mathcal{L}$ denotes the Lebesgue measure on $\R^2$.

\begin{lemma}[{\cite[Lemma 2.3]{trudwang}}]
If $u$ is a $C^2$ convex function, then $$M\!A_u(\omega)=\mathcal{L}(Du(\omega))=\int_\omega (\det D^2 u)d\mathcal{L}\,.$$
\end{lemma}

\begin{defi}
Given a nonnegative measure $\nu$ on $\Omega$, we say a convex function $u:\Omega\rar\R$ is a generalized solution to the Monge-Amp\`ere equation
\begin{equation} \label{general monge ampere}
\det D^2u=\nu
\end{equation}
if $M\!A_u(\omega)=\nu(\omega)$ for all Borel subsets $\omega$. In particular, given an integrable function $f:\Omega\rar\R$, $u$ is a generalized solution to the equation $\det D^2u=f$ if and only if, for all $\omega$, $$M\!A_u(\omega)=\int_\omega fd\mathcal{L}\,.$$
\end{defi}

We collect here, without proofs, some facts which will be used in the following. Unless explicitly stated, the results hold in $\R^n$, although we are only interested in $n=2$. 

\subsection{Stability and comparison principle}

Let us start by the following important lemma, which concerns the continuity of the Monge-Ampère measure.

\begin{lemma}[{\cite[Lemma 2.2]{trudwang}}] \label{convergence of solutions}
Let  $u_n$ be a sequence of convex functions on a convex domain $\Omega$. If $u_n$ converges uniformly on compact sets to $u_\infty$, then the Monge-Amp\`ere measures $M\!A_{u_n}$ converge weakly to $M\!A_{u_\infty}$. 
\end{lemma}

Second, the following comparison principle is the key ingredient, for instance, for every result of uniqueness.

\begin{theorem}[Maximum principle, \cite{trudwang,gutierrez}] \label{comparison principle}
Given a bounded convex domain $\Omega$ and two convex functions $u_+,u_-\in C^0(\overline\Omega)$, if $M\!A_{u_+}(\omega)\leq M\!A_{u_-}(\omega)$ for every Borel subset $\omega$, then
$$\min_{\overline\Omega}(u_+-u_-)=\min_{\partial\Omega}(u_+-u_- )\,.$$  
\end{theorem}

The following is a direct consequence.

\begin{cor}[Comparison principle] \label{cor comparison principle}
Given a bounded convex domain $\Omega$ and two convex functions $u_+,u_-\in C^0(\overline\Omega)$, if $u_+\geq u_-$ on $\partial\Omega$ and $M\!A_{u_+}(\omega)\leq M\!A_{u_-}(\omega)$ for every Borel subset $\omega$, then
$u_+\geq u_-$ on $\Omega$.
\end{cor}

In particular, we have the following result of uniqueness.

\begin{cor} \label{cor comparison principle equality}
Given two generalized solutions $u_1,u_2\in C^{0}(\overline \Omega)$ to the Monge-Amp\`ere equation $\det D^2 u=\nu$ on a bounded convex domain $\Omega$, if $u_1\equiv u_2$ on $\partial\Omega$, then $u_1\equiv u_2$ on $\Omega$.
\end{cor}

\subsection{Existence and regularity}

The following is a classical result of existence for the Dirichlet problem for Monge-Amp\`ere equations.

\begin{theorem}[Dirichlet problem, {\cite[Theorem 1.6.2]{gutierrez}}] \label{thm existence classical monge ampere}
Let $\Omega$ be a bounded strictly convex domain. Given any continuous function $g:\partial\Omega\to\R$ and any Borel measure $\nu$ with $\nu(\Omega)<+\infty$, there exists a generalized solution $u\in C^0(\overline\Omega)$ of the problem
$$
\begin{cases}
\det D^2u=\nu & \text{in }\Omega \\
u|_{\partial\Omega}=g~.
\end{cases}
$$
\end{theorem}

We remark here that Theorem \ref{thm existence classical monge ampere} does not apply directly to Equation \eqref{eq monge ampere curvature}, since in that case the hypothesis of finite total measure is not satisfied. Moreover, the boundary value will not be continuous in the general problem we consider. We also have the following important regularity property:

\begin{theorem}[{\cite[Theorem 3.1]{trudwang}}] \label{solution smooth}
Let $u$ be a strictly convex generalized solution to $\det D^2 u=f$ on a bounded convex domain $\Omega$ with smooth boundary. If $f>0$ and $f$ is smooth, then $u$ is smooth.
\end{theorem}

The following property will be used repeatedly in the paper, and is a peculiar property of dimension $n=2$.

\begin{theorem}[Aleksandrov-Heinz, {\cite[Remark 3.2]{trudwang}}] \label{solution strictly convex dimension 2}
Let $f$ be a positive function and let $u$ be a generalized solution of the Monge-Amp\`ere equation $\det D^2 u =f$ on a domain $\Omega\subset\R^2$. Then $u$ is strictly convex.
\end{theorem}



\subsection{A generalized comparison principle}

In this section we will prove a version of the maximum principle (Theorem \ref{comparison principle}) which we can apply to functions valued in $\R \cup \{+\infty\}$ which satisfy a Monge-Amp\`ere equation on their domain of support. The following definition generalizes Definition \ref{defi gradient surjective}.

\begin{defi} A closed convex function $u$ on $\R^2$ taking values in $\R \cup \{+\infty\}$ is \emph{gradient-surjective} if the sub-differential gives a surjective set-valued map from the interior of its essential domain to $\R^2$.
\end{defi}

\begin{prop}[Generalized comparison principle]\label{prop:genmax}
Suppose $\Omega$ is a convex bounded domain, $u_+: \overline{\Omega} \to \R \cup \{+\infty\}$ is a closed convex function, and $u_- \in C^0(\overline{\Omega})$ is convex.
 If $u_+$ is gradient-surjective and $M\!A_{u_+}(\omega) \leq M\!A_{u_-}(\omega)$ for every Borel subset $\omega \subset \Omega_{u_+}$, then
\[
\min_{\overline{\Omega}}(u_+ - u_-) = \min_{\partial \Omega}(u_+ - u_-)~.
\]
\end{prop}

\begin{proof}
Under the assumptions, the function $u_+ - u_-$ is lower-semicontinuous on $\overline{\Omega}$, so it attains its minimum value at some point $\mathsf x_0 \in \overline{\Omega}$.
We first show that $\mathsf x_0 \notin \partial \Omega_{u_+} \setminus \partial \Omega$.

Indeed, suppose otherwise. Let $p \in \partial_{u_-}(\mathsf x_0)$, and let $l(\mathsf x) = u_-(\mathsf x_0) + p\cdot(\mathsf x - \mathsf x_0)$ be the corresponding affine support. Since $u_+$ is convex, the set $\Omega_{u_+}$ is convex, so it has a supporting hyperplane at $\mathsf x_0$. Let $q$ be the outward normal vector to such a hyperplane, so that the linear function $m(\mathsf x) = q \cdot (\mathsf x-\mathsf x_0)$ is negative on $\Omega_{u_+}$. Let $\tilde{p} = p + q$. Since $l$ is a support for $u_-$, for any $\mathsf x \in \Omega_{u_+}$, we have $l(x) < u_-(x)$, in other words $$\tilde{p}\cdot (\mathsf x - \mathsf x_0) < u_-(\mathsf x) - u_-(\mathsf x_0).$$

Now we use the property that $u_+$ is gradient-surjective to find a point $\mathsf x_1 \in \Omega_{u_+}$ for which $\tilde{p} \in \partial_{u_+}(\mathsf x_1)$. Let $\tilde{l}(\mathsf x)$ be the corresponding affine support for $u_+$ at $\mathsf x_1$, that is $$\tilde{l}(\mathsf x) = u_+(\mathsf x_1) + \tilde{p} \cdot (\mathsf x - \mathsf x_1)~.$$ Using $\tilde{l}(\mathsf x_0) \leq u_+(\mathsf x_0)$, we have $$u_+(\mathsf x_1) - u_+(\mathsf x_0) \leq \tilde{p} \cdot (\mathsf x_1 - \mathsf x_0)~.$$
Since $u_+ - u_-$ is minimized at $\mathsf x_0$, we have $$u_+(\mathsf x_0) - u_-(\mathsf x_0) \leq u_+(\mathsf x_1) - u_-(\mathsf x_1)~.$$ Putting these inequalities together gives
\[
\tilde{p}\cdot (\mathsf x_1 - \mathsf x_0) < u_-(\mathsf x_1) - u_-(\mathsf x_0) \leq u_+(\mathsf x_1) - u_+(\mathsf x_0) \leq \tilde{p} \cdot (\mathsf x_1 - \mathsf x_0)
\]
which is a contradiction. We conclude that $\mathsf x_0 \notin \partial \Omega_{u_+} \setminus \partial \Omega$. 

The rest of the argument is essentially the proof of the standard comparison principle (following \cite[Theorem 1.4.6]{gutierrez}). Suppose that $\mathsf x_0 \in \Omega_{u_+}$ and also for the sake of contradiction that $$u_+(\mathsf x_0) - u_-(\mathsf x_0) < \min_{\partial \Omega}(u_+ - u_-)~.$$ Then it follows also that $$u_+(x_0) - u_-(x_0) < \min_{\partial \Omega_{u_+}}(u_+ - u_-)~,$$ since otherwise the minimum would be attained on $\partial \Omega_{u_+} \setminus \partial \Omega$. By adding a suitable constant to $u_-$, we may arrange that $$u_+(\mathsf x_0) - u_-(\mathsf x_0) < 0 < \min_{\partial \Omega_{u_+}}(u_+ - u_-)~.$$ By replacing $u_-$ with $u_- + \delta |\mathsf x - \mathsf x_0|^2$ for small enough $\delta$, we can preserve these inequalities and also arrange that $M\!A_{u_+}(\omega) < M\!A_{u_-}(\omega)$ with strict inequality.

Let $U = \{\mathsf x \,|\, u_+(\mathsf x) - u_-(\mathsf x) < 0 \}$. A priori, since $u_+$ is only semicontinuous, $U$ need not be open; however, by arrangement $U \subset \Omega_{u_+}$, and $u_+$ is continuous on $\Omega_{u_+}$, so indeed $U$ is open. In fact, the set $\{\mathsf x \,|\, u_+(\mathsf x) - u_-(\mathsf x) \leq 0 \}$ is closed and contained in $\Omega_{u_+}$, so $U$ is compactly contained in $\Omega_{u_+}$, and $u_+$ is continuous on $\overline{U}$. Hence, $u_+ = u_-$ on the boundary of $U$, with $u_+ < u_-$ on the interior. It follows that $\partial_{u_-} (U) \subset \partial_{u_+}(U)$, which contradicts the strict inequality $M\!A_{u_+}(U) < M\!A_{u_-}(U)$.
\end{proof}

The following is a straightforward consequence of the generalized comparison principle.

\begin{prop} \label{prop conv comparison}
Let $\Omega$ be a convex domain. Suppose that $u_+, v: \overline{\Omega} \to \R \cup \{+\infty\}$ are closed convex functions with $u_+$ gradient-surjective, $v \in C^0(\overline{\Omega})$, and $M\!A_{u_+}(\omega) \leq M\!A_v(\omega)$ for every Borel subset $\omega \subset \Omega_{u_+}$. Suppose furthermore that $v(\xi) \leq 0$ at every point $\xi \in \partial \Omega$ for which $u_+(\xi) < +\infty$. Then $$u_+ \geq \mathrm{conv}(u_+|_{\partial \Omega}) + v~.$$
\end{prop}

\begin{proof}
Set $\ph = u_+|_{\partial \Omega}$. By the remark following Definition \ref{defi convex envelope}, it is enough to show that $u_+ \geq l + v$ for every affine function $l$ on $\overline{\Omega}$ with $l|_{\partial \Omega} \leq \ph$. By the assumption on $v$, the restriction of $l + v$ to $\partial \Omega$ is less than or equal to $\ph$, and its Monge-Amp\`{e}re measure coincides with that of $v$. Hence we may apply the generalized comparison principle to conclude $u_+ \geq l + v$.
\end{proof}


\section{Gauss map and minimal Lagrangian maps} \label{sec image of gauss map and support function on boundary}

In this section, we will study some properties of the Gauss map and the support function of future-convex entire surfaces with Gaussian curvature bounded from above and below by positive constants. We thus prove Theorem \ref{prop constant curvature convex hull1}, which is a refined version of Theorem \ref{thm gauss image}. We will then study the relation with minimal Lagrangian maps with values in the hyperbolic plane, and derive Corollary \ref{cor image min lag map}  as a consequence. 



\subsection{Classical barriers} \label{ex support function revolution} 
We give here the construction of some explicit surfaces of constant Gaussian curvature. Besides being examples of the theory previously explained, Example \ref{ex CGC in half-disk} will serve as a \emph{barrier} in the proof of Theorem \ref{prop constant curvature convex hull1} below. 

These surfaces are obtained as \emph{surfaces of revolution}, that is, they are invariant under a 1-parameter group of hyperbolic isometries in $\SO(2,1)<\isom(\R^{2,1})$. Surfaces of this form were studied in \cite{hanonomizu}, where the first examples of non-standard isometric embeddings of $\Hyp^2$ in $\R^{2,1}$ were provided.
Up to conjugation, we can assume the 1-parameter group has the form
\begin{equation} \label{eq 1-parameter hyperbolic}
\left\{\begin{pmatrix} 1 & 0 & 0 \\ 0 & \cosh(s) & \sinh(s) \\ 0 &\sinh(s) & \cosh(s) \end{pmatrix}\,|\,s\in\R\right\}~.
\end{equation}
Hence we consider surfaces $\Sigma$ parameterized by 
\begin{equation}\label{eq 1-parameter parametrization}
(t,s)\mapsto\begin{cases}
x_1(t,s)=g(t) \\ x_2(t,s)=\sinh(s)r(t) \\ x_3(t,s)=\cosh(s)r(t)
\end{cases}~.
\end{equation}
That is, we apply the 1-parameter hyperbolic group to the planar curve $(g(t),0,r(t))$. Following \cite{hanonomizu}, one can assume that
\begin{equation}\label{eq arclength}
g'(t)^2-r'(t)^2=1,
\end{equation}
which means that, for $s=s_0$ fixed, the planar curve $\Sigma\cap\{x_2\cosh(s_0)=x_3\sinh(s_0)\}$ is parameterized by arclength. 

\begin{remark} \label{rmk invariant support function}
Viewing the space $\overline{\D} \times \R$ as the space of achronal planes in $\R^{2,1}$, it is straightforward to write down the action of the 1-parameter group of Equation \eqref{eq 1-parameter hyperbolic} on this space. Using the fact that if $\Sigma$ is invariant under this group then so must be the graph of its support function $u_\Sigma$ in $\overline{\D} \times \R$, it can be shown that $u_\Sigma$ satisfies the following invariance (compare Equation \eqref{eq action on support functions}):
\begin{equation} \label{eq invariant support function}
u_\Sigma(x,y)=\sqrt{1-y^2} \cdot u_\Sigma\left(\frac{x}{\sqrt{1-y^2}},0\right)~.
\end{equation}
As a consequence, if $\xi=(x,y)\in\partial\D$ so that $x^2+y^2=1$, then
$$u_\Sigma(\xi)=|x|\cdot u_\Sigma\left(\frac{x}{|x|},0\right)=\begin{cases} x\cdot u_\Sigma(1,0) & \text{if }x\geq 0 \\ -x\cdot u_\Sigma(-1,0) & \text{if }x<0 \end{cases}~.$$
This function is affine on both half-planes $x \geq 0$ and $x \leq 0$. If $u_\Sigma(1,0) + u_\Sigma(-1,0) = 0$, then the two affine functions agree, and $u_\Sigma|_{\partial \D}$ coincides with support function of the future of a point. If $u_\Sigma(1,0) + u_\Sigma(-1,0) > 0$ then the two affine functions meet at a convex angle, and $u_\Sigma|_{\partial \D}$ coincides with the support function of the future of the segment with end points $(u_\Sigma(-1,0),0,0)$ and $(u_\Sigma(1,0),0,0)$. If $u_\Sigma(1,0) + u_\Sigma(-1,0) < 0$ then $u_\Sigma|_{\partial \D}$ coincides with the support function of the future of a hyperbola given as the intersection of the null cones of two points.

\end{remark}

\begin{example}[Entire CGC surfaces with surjective Gauss map]
The Gaussian curvature of the surface parametrized by \eqref{eq 1-parameter parametrization} assuming \eqref{eq arclength} is given by the simple formula $K(s,t) = r''(t)/r(t)$ \cite[Equation 5]{hanonomizu}. For any $a > 0$, we consider first the solution given by
$$r(t)=a \cosh(t)~,$$
which therefore has $g'(t)=\sqrt{1+a^2\sinh^2(t)}$. By choosing 
$$g(t)=\int_0^t \sqrt{1+a^2\sinh^2(x)}dx~,$$
the corresponding surface (say, $\Sigma_a$) is invariant by the reflection $(x_1,x_2,x_3)\mapsto(-x_1,x_2,x_3)$. When written as a graph $\Sigma_a=graph(f_a)$, $f_a$ has therefore a minimum point at the origin. We remark that, when $a=1$, $\Sigma_1$ is the hyperboloid $\mathrm{Hyp}$.

By multiplying $\Sigma_a$ by the factor $1/\sqrt{K}$, one obtains analogously surfaces $\Sigma_a^K=(1/\sqrt K)\Sigma_a$ of constant Gaussian curvature $K$. To compute the support function $u_{\Sigma_a^K}$ of $\Sigma_a^K = \mathrm{graph}(f_a^K)$, we remark that $u_{\Sigma_a^K}(1,0)$ can be expressed as \cite[Section 2.3]{Bonsante:2015vi}:
$$u_{\Sigma_a^K}(1,0)=\lim_{x_1\to+\infty}(x_1-f^K_a(x_1,0,0))=\lim_{t\to+\infty}(g(t)-r(t))~.$$
It can be thus shown that $F(a):=u_{\Sigma_a^K}(1,0)$ is finite for every $a$, is a decreasing function of $a$, and 
\begin{equation} \label{eq lim F is infinity}
\lim_{a\to 0^+} F(a)=+\infty~.
\end{equation}


Using Remark \ref{rmk invariant support function}, we therefore have, for $\xi=(x,y)\in\partial\D$:
$$u_{\Sigma_a^K}(\xi)=\frac{F(a)}{\sqrt K}|x|~.$$
So, when $a\in(0,1)$, the domain of dependence of $\Sigma_a^K$ is the \emph{future of a segment}. That is, 
$$\mathcal D_{\Sigma^K_a}=\bigcup_{x\in\left[-\frac{F(a)}{\sqrt K},\frac{F(a)}{\sqrt K}\right]} I^+(x,0,0)~.$$
See Figure \ref{fig:finitetrough}.
From the expression \eqref{eq invariant support function} of Remark \ref{rmk invariant support function}, we also see that $u_{\Sigma^K_a}$ is finite on $\overline \D$ and $u_{\Sigma^K_a}\in C^0(\overline\D)$. Moreover, again from \eqref{eq invariant support function} we get:
\begin{equation} \label{eq support function of invariant solutions on vertical line}
u_{\Sigma^K_a}(0,y)=\sqrt{1-y^2}u_{\Sigma^K_a}(0,0)=-\frac{a}{\sqrt K}\sqrt{1-y^2}~,
\end{equation}
which corresponds to the fact that $\Sigma_a^K\cap \{x_1=0\}$ is a hyperbola through the point $(0,0,a/\sqrt{K})$.

\begin{figure}[htb]
\centering
\includegraphics[height=3.9cm]{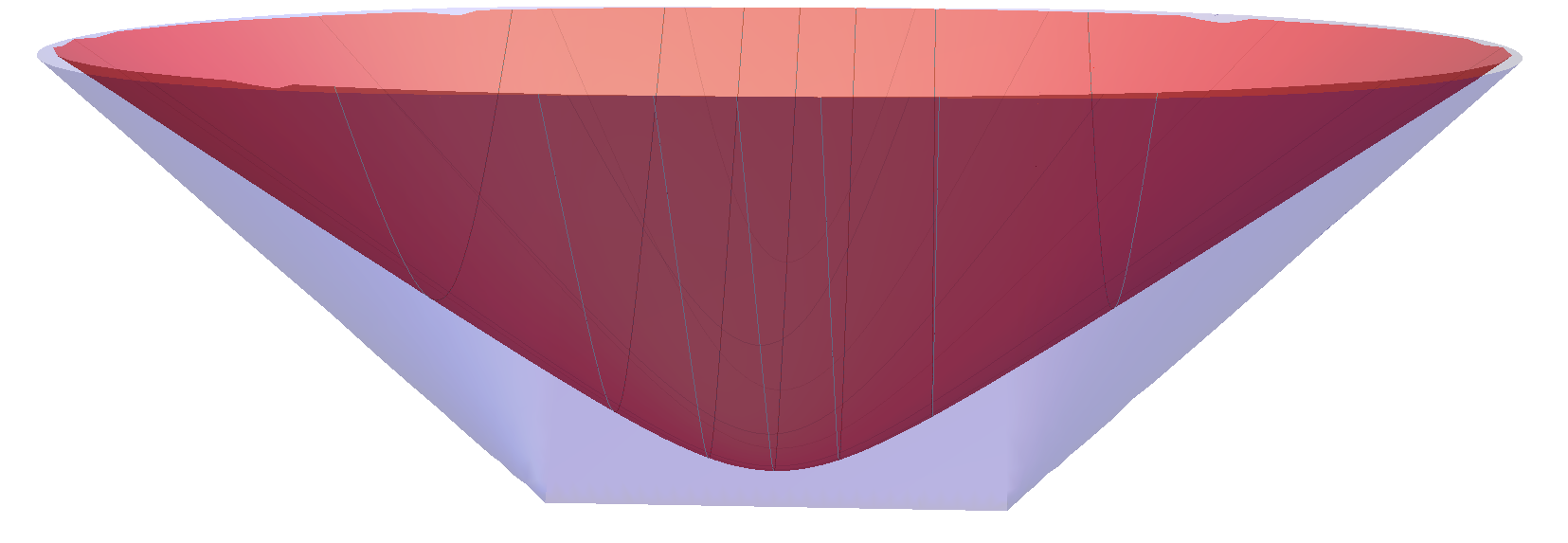}
\caption{The surface of revolution $\Sigma_a^K$, corresponding to the choice $r(t)=a\cosh(t)$. (Here $K=1$ and $a=1/2$.) The domain of dependence is the future of a spacelike segment.  \label{fig:finitetrough}}
\end{figure}

\end{example}

\begin{example}[Entire CGC surfaces with Gauss map to a half-plane] \label{ex CGC in half-disk}
Another useful family of surfaces, still studied in \cite{hanonomizu}, is 
obtained by the choice $r(t)=e^t$. By writing the explicit expression of
$$g(t)=\int_0^t \sqrt{1+r'(x)^2}dx~,$$
this gives:
$$(t,s)\mapsto\begin{cases}
x_1(t,s)=\frac{1}{\sqrt K}\left(\sqrt{1+e^{2t}}-\frac{1}{2}\log\left(\frac{\sqrt{1+e^{2t}}+1}{\sqrt{1+e^{2t}}-1}\right)\right) \\ x_2(t,s)=\frac{1}{\sqrt K}\sinh(s)e^t \\ x_3(t,s)=\frac{1}{\sqrt K}\cosh(s)e^t
\end{cases}$$
Let us call $\Sigma_0^K$ such surface. See also Figure \ref{fig:revolution}. A direct computation, using Remark \ref{rmk invariant support function} shows that the corresponding support function is
\begin{equation}\label{eq chord barrier}
u_{\Sigma_0^K}(x,y)=\begin{cases} -\frac{1}{2\sqrt K}x\log\left(\frac{1+\sqrt{1-\frac{x^2}{1-y^2}}}{1-\sqrt{1-\frac{x^2}{1-y^2}}}\right) & x\geq 0 \\ +\infty & x<0\end{cases}~.
\end{equation}
This is another solution of Equation \eqref{eq monge constant curvature}, which by direct inspection can be shown to be continuous on the closed half-space $\overline{\D_+}$, where $\D_+=\D\cap\{x>0\}$, and $u_{\Sigma_0^K}=0$ on $\partial\D_+$.

\begin{figure}[htb]
\centering
\includegraphics[height=4.3cm]{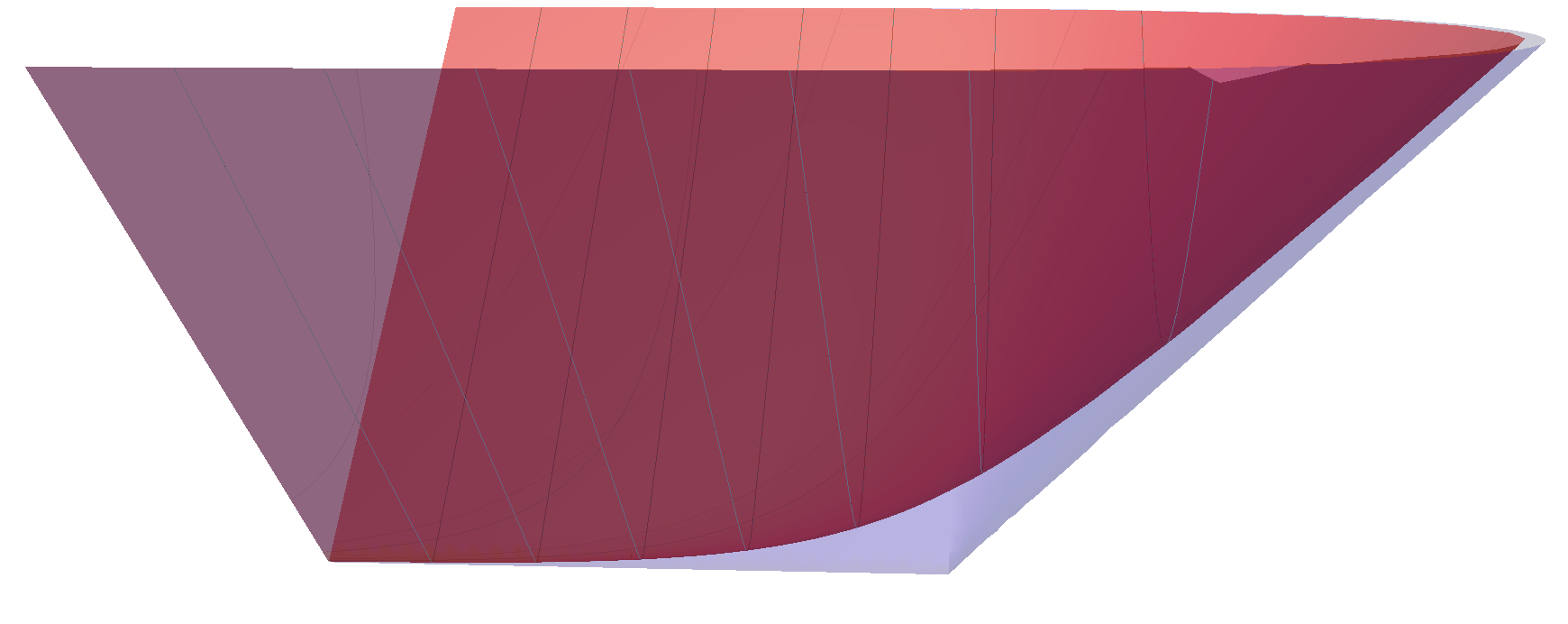}
\caption{The surface of revolution $\Sigma_0^K$, corresponding to the choice $r(t)=e^t$. (Here $K=1$.) The domain of dependence is the future of a half-line.  \label{fig:revolution}}
\end{figure}

\end{example}

\subsection{Image of the Gauss map} \label{subsec image gauss}
We will now prove the following theorem, which is a refined version of Theorem \ref{thm gauss image}, and gives a complete description of the image of the Gauss map of a CGC entire surface in $\R^{2,1}$.

\begin{theorem} \label{prop constant curvature convex hull1} \label{prop constant curvature convex hull2}
Let $\Sigma$ be an entire spacelike surface in $\R^{2,1}$ with Gaussian curvature bounded from above and below by positive constants. Let $u_\Sigma:\D\to\R\cup\{+\infty\}$ be the support function of $\Sigma$. Then
\begin{itemize}
\item The essential domain of $u_\Sigma$, i.e. the set on which $u_\Sigma$ is finite, coincides with the convex hull of $\{\xi\in\partial\D\,|\,u_\Sigma(\xi)<+\infty\}$.
\item For every segment of $\partial\Omega_\Sigma$ with endpoints $\xi_1,\xi_2\in\partial\D$, $u_\Sigma$ restricted to the chord $[\xi_1, \xi_2]$ is the convex envelope of $u_\Sigma|_{\{\xi_1,\xi_2\}}$.
\end{itemize}
\end{theorem}

The second bullet point means that if either $u_\Sigma(\xi_1)$ or $u_\Sigma(\xi_2)$ is infinite then $u_\Sigma$ is infinite on the open chord, and otherwise it is the unique affine function interpolating the values at the endpoints (compare the comment following Proposition \ref{prop reg domain conv}).


\begin{proof}[Proof of Theorem \ref{prop constant curvature convex hull1}]

Let $\mathcal{C}$ be the convex hull of $\{\xi\in\partial\D\,:\,u_\Sigma(\xi)<+\infty\}$. Let $K_0$ be a positive lower bound for the curvature of $\Sigma$. Let $$v(\mathsf z) = -\frac{1}{\sqrt{K_0}}\sqrt{1-|\mathsf z|^2}$$
be the support function of the hyperboloid $\mathrm{Hyp}^{K_0}$. Then $M\!A_{u_\Sigma}(\omega) \leq M\!A_v(\omega)$ for all Borel subsets $\omega \subset \Omega_{u_\Sigma}$ and $v$ is continuous on the closed disk and equal to 0 on the boundary. Hence by Proposition \ref{prop conv comparison}, we have $u_\Sigma \geq \mathrm{conv}(u_\Sigma|_{\partial \D}) + v$. By the remark following Proposition \ref{prop reg domain conv} the essential support of $\mathrm{conv}(u_\Sigma|_{\partial \D})$ is equal to $\mathcal C$. Since $v$ is finite everywhere, this shows that $u_\Sigma$ is infinite at every point outside $\mathcal{C}$. Since $u_\Sigma$ is convex, so is its essential domain. Therefore the essential domain of $u_\Sigma$ is exactly $\mathcal C$. This proves the first bullet point as well as the second bullet point in the case where $u_\Sigma$ is infinite at either of the two endpoints $\xi_i$.

To complete the proof of the theorem, we need only consider the case where both $u_\Sigma(\xi_1)$ and $u_\Sigma(\xi_2)$ are finite for a segment $[\xi_1,\xi_2]$ of $\partial \Omega_\Sigma$. Up to composing $\Sigma$ with an isometry of $\R^{2,1}$, we can assume $\xi_1=(0,-1)$, $\xi_2=(0,1)$ and that $\mathcal C$ is contained in $\{x\geq 0\}$. 
We will show that $u_\Sigma(0,y) = \mathrm{conv}(u_\Sigma|_{\partial\D})(0,y)$ for every $y \in [-1,1]$. 

Let $\Sigma^{K_0}_0$ be the function constructed in Example \ref{ex CGC in half-disk}, whose support function $u_0$ is a solution to the Monge Amp\`{e}re equation \eqref{eq monge constant curvature} on the right half-disk $\D_+$ with $u_0|_{\partial \D_+} = 0$. Let $A_\epsilon$ be the linear hyperbolic transformation of length $\epsilon$ with attracting fixed point $(-1,0)$ and repelling fixed point $(1,0)$. Let $u_\epsilon$ be the support function of $A_\epsilon(\Sigma_0^{K_0})$. Explicitly, in coordinates $(x,y)$ on the disk \cite[Lemma 3.4]{bonfill},
\begin{equation} \label{eq action on support functions}
u_\epsilon(x,y) = (\cosh(\epsilon)+ x\sinh(\epsilon))\ u_0\!\left(\frac{x\cosh(\epsilon)+ \sinh(\epsilon)}{\cosh(\epsilon)+ x\sinh(\epsilon)},\frac{y}{\cosh(\epsilon)+ x\sinh(\epsilon)}\right)~.
\end{equation}
Observe that $u_\epsilon$ is equal to zero on the boundary of the half disk $\D_+^\epsilon$ bounded by the chord $$[(-\tanh(\epsilon),\mathrm{sech}(\epsilon)),(-\tanh(\epsilon),-\mathrm{sech}(\epsilon))]$$
By Proposition \ref{prop conv comparison} applied to $\D_+^\epsilon$, we have that $$u_\Sigma \geq \mathrm{conv}(u_\Sigma|_{\partial \D_+^\epsilon}) + u_\epsilon \quad \textrm{for all $\epsilon$.}$$ Since $u_\Sigma$ is equal to $+\infty$ on the left half-disk, in fact $\mathrm{conv}(u_\Sigma|_{\partial \D_+^\epsilon}) = \mathrm{conv}(u_\Sigma|_{\partial \D})$. Now we take the limit as $\epsilon \to 0$ and use the continuity of $u_0$ on $\overline{\D}_+$ to conclude that $$u_\Sigma \geq \mathrm{conv}(u_\Sigma|_{\partial \D}) + u_0~.$$
Since $u_0$ is zero on the $y$-axis, we conclude that  $u_\Sigma = \mathrm{conv}(u_\Sigma|_{\partial\D})$ on the $y$ axis. The other inequality follows from convexity of $u_\Sigma$.
\end{proof}

\begin{remark} The reason why in the last part of the proof of the previous theorem we did not apply directly Proposition 3.12  to the domain $\mathbb D_+$ is that
this would lead to the following inequality $u_\Sigma\geq \mathrm{conv}(u_\Sigma|_{\partial \mathbb D_+})+u_0$.
However, since the restriction $u_\Sigma$ to $\partial\mathbb D_+\setminus\partial\mathbb D$ is not constantly $+\infty$, 
we can no longer  argue that $\mathrm{conv}(u_\Sigma|_{\partial \mathbb D_+})$ coincides with $\mathrm{conv}(u_\Sigma|_{\partial \mathbb D})$.
So the previous  estimate is not useful to control $u_\Sigma$ a priori on the $y$-axis.
\end{remark}

We then have the following corollary of the results of the previous subsection:
\begin{repcor}{cor image min lag map} 
Let $F:\Hyp^2\to\Hyp^2$ be a minimal Lagrangian map. Then the image $F(\Hyp^2)$ coincides with the interior of the convex hull of 
$\overline{F(\Hyp^2)}\cap \partial\Hyp^2$.
\end{repcor}
\begin{proof}
By Lemma \ref{rmk min lag realizable}, $F$ can be realized as the Gauss map of a CGC-$K$ surface $\Sigma$ in $\R^{2,1}$, which is entire. Indeed the first fundamental form coincides with the metric on the source, and is therefore complete. Therefore $\Sigma$ is entire by Remark \ref{rmk complete implies entire}. Hence, by applying again Theorem \ref{prop constant curvature convex hull1}, the image of the Gauss map of $\Sigma$ (which is identified with $F$) coincides with the convex hull of $\overline{F(\Hyp^2)}\cap \partial\Hyp^2$.
\end{proof}

\section{The case of the ideal triangle} \label{sec triangular surfaces}

In this section, we will consider the special case of the regular domain $\mathcal D_\varphi$ in $\R^{2,1}$, where $\varphi$ takes finite value on precisely three distinct points $\xi_1,\xi_2,\xi_3$ of $\partial \D$, and $\varphi(\xi)=+\infty$ otherwise. This regular domain $\mathcal D_\varphi$ is the intersection of three half-spaces bounded by null planes. We call such a regular domain a triangular domain (See Figure \ref{fig:triangulardomain}).

The purpose of this section is to prove the following:

\begin{prop} \label{prop foliation triangular domain}
Let $K \in (0,+\infty)$ and let $\ph:\partial \D\to\R\cup\{+\infty\}$ be a function with $\ph(\xi_1),\ph(\xi_2),\ph(\xi_3)<+\infty$ and $\ph(\xi)=+\infty$ otherwise. Then the domain $\mathcal D_\ph$ is the domain of dependence of an entire CGC-$K$ surface.  
\end{prop} 

The linear isometry group $\SO_0(2,1)$ acts simply transitively on triples of null planes intersecting at the origin. Furthermore, any triple of nonparallel null planes intersect at a point. Therefore, up to the action of the \emph{affine} isometry group of $\R^{2,1}$, all triangular domains are equivalent to each other. Therefore, it is enough to produce a CGC-$K$ surface in any single triangular domain $\mathcal D_\ph$. Moreover, using also homethety, it suffices to produce a CGC-1 surface with this property.

\begin{figure}[htb]
\centering
\includegraphics[height=4cm]{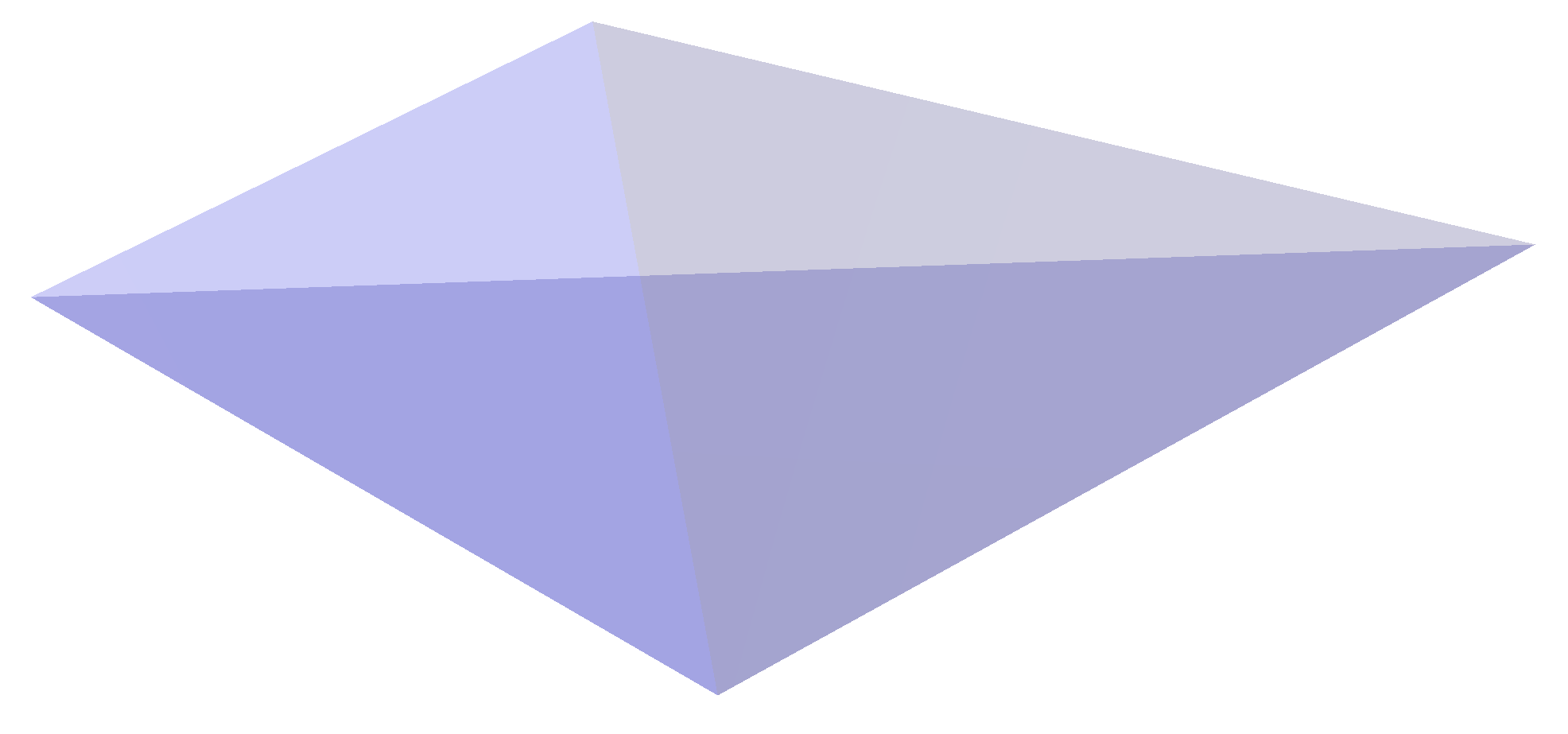}
\caption{When $\varphi$ is finite on exactly three points, the domain of dependence  $\mathcal D_\varphi$ is the intersection of the future of three null planes. \label{fig:triangulardomain}}
\end{figure}

\subsection{Harmonic maps to an ideal triangle}

We will construct such a surface $\Sigma$ by using the correspondence between minimal Lagrangian maps and CGC-K surfaces described in Lemma \ref{rmk min lag realizable}. In fact, we will construct a minimal Lagrangian map by way of harmonic maps.

From the classical theory of harmonic maps, a harmonic map $f$ from $\C$ to $\Hyp^2$ is determined up to isometries of $\Hyp^2$ by the Hopf differential $\Phi_f$ and the holomorphic energy density $\mathscr H_f$, which are defined by
\begin{align*}
\Phi_f &= (f^*h_{\Hyp^2})^{2,0} = \langle \partial f, \overline{\partial} f\rangle~,\\
\mathscr{H}_f &=||\partial f||^2 = \langle \partial f, \partial f\rangle~,
\end{align*}
where we use the decomposition $df=\partial f+\overline\partial f$. It is well-known that $\Phi_f$ is a holomorphic quadratic differential. Setting $\mathscr{H}_f = e^{2h}$ and $\Phi_f = \phi(\mathsf z)d\mathsf z^2$, the function $h$ satisfies the Bochner equation:
\begin{equation} \label{eq bochner}
\Delta h=e^{2h}-\frac{|\phi|^2}{e^{2h}}~.
\end{equation}



\begin{prop}[{\cite{labourieCP}}] \label{prop min lag decomposition}
Given two hyperbolic surfaces $(S,h)$ and $(S',h')$, a diffeomorphism $F:(S,h)\to (S',h')$ is minimal Lagrangian if and only if there exist harmonic diffeomorphisms $f:(S_0,X_0)\to (S,h)$ and $f':(S_0,X_0)\to (S',h')$, where $(S_0,X_0)$ is a Riemann surface, such that:
\begin{enumerate}
\item $F=f'\circ f^{-1}$,
\item $\Phi_f=-\Phi_{f'}$, and
\item $\mathscr{H}_f = \mathscr{H}_{f'}$.
\end{enumerate}
\end{prop} 

If there exists a CGC-1 surface $\Sigma$ whose domain of dependence is a triangular domain, then by Theorem \ref{thm gauss image} the image of the Gauss map $G_{\Sigma}:\Sigma\to\Hyp^2$ must be an ideal triangle in $\Hyp^2$. From the results of \cite{hantamtreibergswan}, it is known that if an harmonic map $f:\C\to\Hyp^2$ has polynomial Hopf differential of degree $n$, then its image is an ideal polygon with $n+2$ vertices. Hence our strategy is to consider a minimal Lagrangian diffeomorphism $F_0=f_0'\circ f_0^{-1}$, where $f_0,f_0':\C\to\Hyp^2$
are harmonic maps with Hopf differentials
$$\Phi_0=\mathrm{Hopf}(f_0)=-\mathsf zd\mathsf z^2\qquad \Phi'_0=\mathrm{Hopf}(f'_0)=\mathsf zd\mathsf z^2~.$$
There is a natural choice of solution to the Bochner equation on $\C$ with Hopf differential $\pm\Phi_0$, given by the following result:

\begin{theorem}[{\cite{wanau}}] \label{thm wanau}
Let $\Phi$ be a holomorphic quadratic differential on $\C$ which is not identically zero. Then there exists a unique smooth function $h:\C\to\R$ which solves
$$\Delta h=e^{2h}-|\phi|^2e^{-2h}~,$$
such that $e^{2h}-|\phi|^2e^{-2h}>0$ and the Riemannian metric $e^{2h}|d\mathsf z|^2$ is complete.
\end{theorem}

\begin{remark} \label{rmk CMC} Associated to the harmonic map $f'_0$ is also a constant mean curvature spacelike immersion $\sigma_H: \C \to \R^{2,1}$, which is conformal and also has $f'_0$ as its Gauss map (see for example \cite{choitreibergs}). The induced metric on $\sigma_H(\C)$ is $e^{2h}|d\mathsf z|^2$. Therefore, the solution $h$ given by Theorem \ref{thm wanau} has the property that $\sigma_H(\C)$ is complete, and hence properly embedded. Choosing a normalization so that the mean curvature of $\sigma_H$ is 1/2, the CGC-1 immersion $\sigma$ is given by the formula $\sigma(\mathsf z) = \sigma_H(\mathsf z) - G(\sigma_H(\mathsf z))$, where $G$ is the Gauss map of $\sigma_H$, taking values in $\Hyp^2 \subset \R^{2,1}$. This classical observation also holds in Euclidean space. In a future paper, we will use it to derive similar results for constant mean curvature surfaces.
\end{remark}

Let $h_0$ be the solution of the Bochner equation with Hopf differential $\pm\Phi_0$ guaranteed by Theorem \ref{thm wanau}. By Proposition \ref{prop min lag decomposition}, this determines up to isometry a minimal Lagrangian diffeomorphism $F_0=f_0'\circ f_0^{-1}$, where $f_0,f_0':\C\to\Hyp^2$. By Lemma \ref{rmk min lag realizable}, this in turn determines up to isometry a CGC-1 spacelike immersion $\sigma_0: \C \to \R^{2,1}$. According to the proof of Lemma \ref{rmk min lag realizable}, the immersion data $(\I,B)$ of $\sigma_0$ are uniquely determined by
\begin{equation} \label{eq I and B}
\begin{split}
\I &= f_0^*h_{\Hyp^2}~, \\
(f'_0)^*h_{\Hyp^2}(\cdot, \cdot) &= f_0^*h_{\Hyp^2}(B\cdot, B\cdot)~, \\
\end{split}
\end{equation}
and the condition that $B$ is positive and symmetric for $f_0^*(h_{\Hyp^2})$. In the following, we will express $\I$ and $B$ explicitly.

\subsection{An expression for the embedding data}\label{subsec comp embedding data}

We will ultimately show that the CGC-1 surface $\sigma_0(\C)$ is entire; for this, we will need to analyze the asymptotic behavior of $\sigma_0$ and in particular of the function $h_0$. To this end, it is useful to introduce the local chart $$\mathsf w=(2/3)\mathsf z^{3/2}~.$$ This means that $\mathsf w$ is a branch of square root of $\mathsf z^3$, up to the factor $2/3$. We remark that $\mathsf w$ gives a chart on any sector of angle less than $4\pi/3$. Since $\Phi_0=-\mathsf zd\mathsf z^2$ has an order 3 rotational symmetry, and the uniqueness part of Theorem \ref{thm wanau} implies that $h_0$ has the same symmetry, the parameter $\mathsf w$ will be sufficient to understand the whole geometry of the problem. 

\begin{remark}In fact, $h_0$ is totally rotationally symmetric since the magniture $|\phi|^2$ is the only contribution to the Bochner equation. Even though this remark is not strictly necessary for any of our results, it is worth pointing out that it reduces the construction of $\sigma_0$ to the solution of an \emph{ordinary} differential equation.
\end{remark}

We now give expressions for the embedding data $(\I,B)$ of $\sigma_0$ in a $\mathsf w$ coordinate chart. First note that with respect to this coordinate $\Phi_0=-d\mathsf w^2$. Moreover, the logarithmic holomorphic energy density $\tilde{h}_0$ with respect to the $\mathsf w$ chart is related to $h_0$ by
$$\tilde{h}_0 = h_0 - \frac{1}{2} \log |\mathsf z|$$
and by Equation \eqref{eq bochner} with $\phi = 1$ it satisfies the Bochner equation $\Delta \tilde{h}_0 = 2 \sinh(2\tilde{h}_0)$, where the Laplacian is with respect to the metric $|d\mathsf w|^2$.

The first fundamental form is by construction
$$\I = f_0^*h_{\Hyp^2} = - d\mathsf w^2 + e|d\mathsf w|^2 -  d \overline{\mathsf w}^2~,$$
where $e$ is the energy density of $f_0$ with respect to the flat metric $|d \mathsf w|^2$. Then, using the equations $e=\mathscr H+\mathscr L$, where $\mathscr L$ is the anti-holomorphic energy density, and in the $\mathsf w$ coordinate $\mathscr H\!\mathscr L=|\phi|^2 = 1$, we have
$$
e = e^{2\tilde{h}_0}+ e^{-2\tilde{h}_0} = 2 \cosh(2\tilde{h}_0),
$$
where $\tilde{h}_0$ is the logarithmic holomorphic energy density in the $\mathsf w$ coordinate. Similarly, the third fundamental form is given by
$$\III = (f'_0)^*h_{\Hyp^2} = d\mathsf w^2 + 2\cosh(2\tilde{h}_0)|d\mathsf w|^2 + d \overline{\mathsf w}^2~.$$
To write $B$ in coordinates, it is helpful to introduce the coordinates $\mathsf w=u+iv$, so that $d\mathsf w^2 + d \overline{\mathsf w}^2 = 2(du^2 - dv^2)$. Then we obtain:

\begin{equation} \label{eq first fund form w}
\begin{split}
\I = f_0^*h_{\Hyp^2}&=2\cosh(2\tilde{h}_0)(du^2+dv^2)-2(du^2-dv^2) \\
&=(2\sinh(\tilde{h}_0))^2du^2+(2\cosh(\tilde{h}_0))^2dv^2
\end{split}
\end{equation}
and similarly
\begin{equation}\label{eq third fund form w}
\begin{split}
\III = (f'_0)^*h_{\Hyp^2}&=2\cosh(2\tilde{h}_0)(du^2+dv^2)+2(du^2-dv^2) \\
&=(2\cosh(\tilde{h}_0))^2du^2+(2\sinh(\tilde{h}_0))^2dv^2~.
\end{split}
\end{equation}
Now it is easy to see that from Equation (\ref{eq I and B}) that
\begin{equation} \label{eq shape op w}
B=\coth(\tilde{h}_0)du\otimes \frac{\partial}{\partial u} + \tanh(\tilde{h}_0)dv\otimes \frac{\partial}{\partial v}~.
\end{equation}

\subsection{A priori estimates for Bochner equation}

In this section we provide the estimates for $\tilde{h}_0$ which will allow us to conclude that $\sigma_0: \C \to \R^{2,1}$ is a proper embedding.
Let $r = |\mathsf w|$ be the radial coordinate with respect to the $\mathsf w$ chart. A particular case of \cite[Lemma 1.2]{hantamtreibergswan} provides an a priori bound on such rotationally invariant solution. 

\begin{lemma} \label{lemma a priori1 bochner} There exist constants $C>0$ and $r_0>0$ such that 
$$0\leq \tilde{h}_0(\mathsf w) \leq e^{-Cr}~,$$
as long as $|\mathsf w|\geq r_0$.
\end{lemma}

\begin{remark}
Recalling the expression (\ref{eq shape op w}) for the shape operator $B$, we see that the principal curvatures of $\sigma_0$,
$$\lambda=\coth(\tilde{h}_0)\qquad \mu=\tanh(\tilde{h}_0)~,$$
satisfy $\lambda\to + \infty$ and $\mu\to 0$ as $|\mathsf w|\to+\infty$.
\end{remark}

We will actually need a similar bound, but from below, on the function $\tilde{h}_0$, which we prove in the following lemma.

\begin{lemma} \label{lemma a priori2 bochner}
There exist constants $A>0$ and $r_0>0$ such that
$$\tilde{h}_0(\mathsf w)\geq \frac{A}{\sqrt{r}}e^{- 2r}~,$$
as long as $|\mathsf w|\geq r_0$.
\end{lemma}
\begin{proof}
Recall that for $r > 0$, $\tilde{h}_0$
solves the PDE
$$\Delta \tilde{h}_0=2\sinh(2\tilde{h}_0)~.$$
Now, consider the function 
$$v= v(r)=\frac{A}{\sqrt{r}}e^{-2r}~.$$
Since $v(r)$ is rotationally symmetric, we have
$$\Delta v = v''(r) + \frac{v'(r)}{r}$$ 
and by a direct computation we see
$$\Delta v = \left(4 + \frac{1}{4r^2}\right)v~.$$

On the other hand, for fixed $A$, if $r$ is large enough, $v(r)$ is smaller than any power of $1/r$. Therefore,using the Taylor expansion of the hyperbolic sine near zero, for large enough $r$, independent of $A$ so long as $A < 1$ say:
\begin{equation} \label{eq taylor sinh}
2\sinh 2 v\sim 4v+\frac{8}{3}v^3 < \left(4 + \frac{1}{4r^2}\right)v~.
\end{equation}
Hence there exists $r_0$ independent of $A$ so long as $A < 1$, such that 
$$\Delta v>2\sinh 2 v$$
for every $r\geq r_0$.

To conclude, choose $1>A>0$ such that 
$v(r_0) = Ar_0^{-1/2}e^{-2r_0}<\tilde{h}_0(\mathsf w)$ for all $|\mathsf w| = r_0$. Then by the maximum principle for this choice of $A$, we have $\tilde{h}_0 \geq v$ whenever $r\geq r_0$. Indeed, $\tilde{h}_0 - v> 0$ on the circle $\{r=r_0\}$ by the choice of $A$, and $\tilde{h}_0- v\to 0$  as $r\to +\infty$ since each function goes to 0, so if we suppose that the set $\{\tilde{h}_0<v\}\cap \{r\geq r_0\}$ is non-empty, then $\tilde{h}_0 - v$ has to assume a negative minimum value. But at the minimum point, 
$$\Delta \tilde{h}_0=2\sinh 2\tilde{h}_0< 2\sinh 2v<\Delta v~,$$
hence $\Delta (\tilde{h}_0- v)<0$ and this gives a contradiction. Hence we conclude that
$$\tilde{h}_0> v=\frac{A}{\sqrt{r}}e^{-2r}$$
for $r\geq r_0$ and for a suitable choice of $A>0$, as claimed.
\end{proof}

\subsection{Proof of Proposition \ref{prop foliation triangular domain}} 

We begin with several lemmas, which will help us to understand the behavior of the immersion $\sigma_0$ in specific directions. The ultimate goal is to show that $\sigma_0$ is a proper embedding. First, we argue that since $\sigma_0$ is convex, it is at the very least a subset of an entire \emph{achronal} surface.

\begin{lemma} \label{lem achronal extension}  Let $\sigma: S \to \R^{2,1}$ be a $C^2$ immersion with everywhere positive definite second fundamental form. Let $G: S \to \Hyp^2$ be the Gauss map of $\sigma$, and suppose that $G$ is injective with convex image. Then $\sigma$ is an embedding and moreover there exists a convex achronal entire surface $\Sigma$ such that the normal of each support plane of $\Sigma$ is in $\overline{G(S)}$, and $\sigma(S)$ is the subset of $\Sigma$ whose support planes have normal contained in $G(S)$.
\end{lemma}

\begin{proof}
The strategy of the proof is to construct the support function of $\Sigma$ from the immersion $\sigma$. Let $\pi: \Hyp^2 \to \D $ be the radial projection to the disk at height 1. Define the function $u_\sigma: \mathrm{Im}(\pi \circ G) \to \R$ by 
$$ u_\sigma(\pi \circ G(\pmb p)) = \langle \pmb p, (\pi \circ G(\pmb p),1)\rangle$$


The fact that the second fundamental form of $\sigma$ is positive definite implies that the function $u_\sigma$ is convex, by straightforward calculation of the Hessian of $u_\sigma$. Since the domain $\pi \circ G(S)$ of $u_\sigma$ is also convex, its convex hull $\mathrm{conv}(u_\sigma)$ is equal to $u_\sigma$ on $\pi \circ G(S)$ and is equal to $+\infty$ on $\overline{\pi \circ G(S)}^c$. By Proposition \ref{prop: convex duality}, $\mathrm{conv}(u_\sigma)$ is dual to an entire achronal surface $\Sigma$. Moreover the Legendre transform gives a homeomorphism from the image of $\pi \circ G$ to the subset of $\Sigma$ consisting of points whose support plane has normal contained in the image of $G$. Composing the Legendre transform with $\pi \circ G$, we obtain an embedding $\sigma': S \to \R^{2,1}$. By construction, the Gauss map of $\sigma'$ is equal to $G$ and the support function of $\sigma'$ is equal to $u_\sigma$ (compare Equation \eqref{eq support function of convex surface}).

Since $\sigma$ is an immersion with positive definite second fundamental form, it is locally the graph of a convex function. By the local nature of the formula to recover a surface from its support function (Equation \eqref{eq inverse support function}), the immersions $\sigma$ and $\sigma'$ must agree. Therefore $\sigma$ is an embedding and its image is exactly those points of $\Sigma$ whose support plane has normal in $G(S)$.
\end{proof}

We will take advantage of the global symmetries of the CGC surface $\sigma_0(\C)$ and its achronal extension $\Sigma$ as in Lemma \ref{lem achronal extension}.

\begin{lemma} \label{lemma symmetry Hsurface}
There exists a dihedral group $\Gamma<\isom(\R^{2,1})$ of order 6 which leaves the surface $\Sigma$ invariant. Moreover, $\Gamma$ is generated by a linear elliptic isometry in $\SO_0(2,1)$ of order 3, and by a reflection in a timelike plane.
\end{lemma}
\begin{proof}
By Theorem \ref{fund theorem}, the surface $\sigma_0(\C)$ is determined up to a global isometry by the embedding data, namely the first fundamental form $\I$ of Equation \eqref{eq first fund form w} and the shape operator $B$ of Equation \eqref{eq shape op w}.

Since the solution $h_0$ to Bochner equation is rotationally invariant, and the holomorphic quadratic differentials $\pm \mathsf zd\mathsf z^2$ have an order 3 rotational symmetry, the embedding data $(\I,B)$ have a dihedral group of (intrinsic) isometries generated by the rotation $\alpha:\mathsf z\mapsto \omega \mathsf z$ (where $\omega$ is a cubic root of the identity) and by the conjugation $\beta:\mathsf z\mapsto\overline {\mathsf z}$. By the uniqueness part of Theorem \ref{fund theorem}, the embeddings $\sigma_0$ and $\sigma_0\circ \alpha$ differ by a global isometry $A\in\isom(\R^{2,1})$. Such $A$ must necessarily preserve orientation and time-orientation, and fix the point $\sigma_0(0)$ and the normal vector $N(\sigma_0(0))$. Hence $A\in \SO_0(2,1)$ is a rotation, and has order three by a similar argument. 

Analogously, one shows that $\sigma_0$ and $\sigma_0\circ \beta$ differ by a time-orientation preserving and orientation-reversing isometry $B$ which fixes the geodesic $\sigma_0(\{\mathrm{Im}(\mathsf z)=0\})$ of $\sigma_0(\C)$ pointwise. We have thus obtained a representation of the dihedral group $\langle \alpha,\beta\rangle$ of order 6 in $\isom(\R^{2,1})$, whose image leaves $\sigma_0(\C)$ invariant.

Since Lemma \ref{lem achronal extension} defines $\Sigma$ canonically in terms of the embedding $\sigma_0$, it is also invariant under the same group of isometries.
\end{proof}

Finally, we show that the surface $\sigma_0(\C)$ looks like a properly embedded surface along its planes of symmetry. First, the following general lemma characterizes properly embedded spacelike curves contained in a timelike plane in terms of their speed and curvature.

\begin{lemma}\label{lem integral of curvature} Let $\gamma: [0, + \infty) \to \R^{1,1}$ be a spacelike curve with curvature $\kappa:[0,+\infty) \to \R_+$ and speed $\nu:[0,+\infty) \to \R_+$. Assume that $$\int_0^{\infty}\exp\left(\int_0^r \kappa(s)\nu(s)ds\right) \nu(r)dr = + \infty~.$$ Then $\gamma$ is proper.
\end{lemma}

\begin{proof} Let $T: [0, + \infty) \to \R^{1,1}$ be the unit tangent vector field along $\gamma$ and $N$ be the normal vector. Denote by $t = t(r)$ the arclength, so that $$\frac{dt}{dr} = \nu(r)~.$$ We have that $\langle T,T \rangle = 1$, $\langle N,N\rangle = -1$, and $\langle T,N \rangle = 0$. Moreover,
$$
\begin{cases}
\frac{dT}{dt} &= \kappa N \\
\frac{dN}{dt} &= \kappa T \\
\end{cases}~.
$$
Observe that the lightlike directions of $T+N$ and $T-N$ are fixed. Let $\xi_-$ be a future-pointing lightlike vector parallel to $T-N$, and consider the function $\rho = \langle \gamma, \xi_- \rangle$. As $\frac{d\gamma}{dt} = T$, we have
$$
\begin{cases}
\frac{d\rho}{dt} &= \langle T, \xi_-\rangle \\
\frac{d^2\rho}{dt} &= \langle \kappa N, \xi_- \rangle = \kappa \langle T, \xi_- \rangle = \kappa \frac{d\rho}{dt}
\end{cases}
$$
where we have used that $\langle T-N, \xi_- \rangle = 0$. So $\frac{d\rho}{dt}(t) = \frac{d\rho}{dt}(0) \exp\left(\int_0^t \kappa(r(\tau)) d\tau \right)$ or in terms of the parameter $r$, $$\frac{d\rho}{dt}(t(r)) = C \exp\left(\int_0^r \kappa(s)\nu(s)  ds \right)~.$$ So
$$\frac{d\rho}{dr}(r) = \frac{d\rho}{dt}(t(r))\cdot \nu(r) =  C \exp\left(\int_0^r \kappa(s) \nu(s)ds\right) \nu(r)$$
and by the assumption,
$$\int_0^{+\infty} \frac{d\rho}{dr} dr = + \infty ~.$$
Therefore $\langle \gamma(r), \xi_-\rangle \to + \infty$ and this implies that $\gamma$ is proper.
\end{proof}

We now apply this general result to the intersection of $\sigma_0(\C)$ with its planes of symmetry. One of these three identical curves is the fixed points of the conjugation $\beta: \mathsf z \mapsto \overline{\mathsf{z}}$.

\begin{cor} \label{cor proper axis} The restriction of $\sigma_0(\mathsf z)$ to the real axis in the $\mathsf z$ coordinate is a proper spacelike curve.
\end{cor}

\begin{proof} Since this curve is fixed by the conjugation symmetry $\beta$, it must be contained in the timelike plane fixed by the corresponding reflection of $\R^{2,1}$. Identify this timelike plane with $\R^{1,1}$. We check properness at each end in turn. For each case, we choose a branch $\mathsf w = \frac{2}{3} \mathsf z^{3/2}$ as above. The negative real axis in the $\mathsf z$ chart corresponds to the imaginary axis $\mathsf w = 0 + iv$ in the $\mathsf w$ chart. By Equation (\ref{eq first fund form w}) we see that the metric on this ray is $(2 \cosh(\tilde{h}_0))^2dv^2$ which always larger than $4dv^2$, and so this end is complete. Since it is contained in a timelike plane, completeness implies that it is properly embedded.

The harder case is the positive real axis, which corresponds to the real axis $\mathsf w = u + 0i $ in the $\mathsf w$ chart. Here the induced metric is $(2 \sinh(\tilde{h}_0))^2du^2$, which tends to 0 as $u \to \infty$. With respect to the $u$ coordinate, Equations (\ref{eq first fund form w}) and (\ref{eq shape op w}) show that the velocity is $\nu = 2 \sinh(\tilde{h}_0(u))$ and the curvature is $\kappa = \coth(\tilde{h}_0(u))$. Therefore
$$\exp\left(\int_0^u \kappa(s)\nu(s)ds\right) = \exp\left(\int_0^u 2 \cosh(\tilde{h}_0(s))ds\right) \geq \exp(2u)$$ 
and using $\sinh \tilde{h}_0 \geq \tilde{h}_0$ together with the estimate of Lemma \ref{lemma a priori2 bochner},
$$ \int_{u_0}^\infty \exp\left(\int_0^u \kappa(s)\nu(s)ds\right) \nu(u)du \geq \int_{u_0}^\infty \exp(2u) \left(\frac{A}{\sqrt{u}}\exp(-2u)\right) \to + \infty~.$$
Hence we conclude from Lemma \ref{lem integral of curvature} that this end of the curve is also proper.
\end{proof}

Now we conclude the proof of the main result of this section.

\begin{proof}[Proof of Proposition \ref{prop foliation triangular domain}]

Let $\sigma_0$ be the CGC-1 immersion constructed in the previous section, which is actually an embedding by Lemma \ref{lem achronal extension}, and let $\Sigma$ be the convex achronal entire surface constructed from $\sigma_0$ via Lemma \ref{lem achronal extension}. We will show that the image $\sigma_0(\C)\subseteq\Sigma$ coincides with $\Sigma$, and this will show that the image of $\sigma_0$ is entire.

Since the image of $\sigma_0$  is open and nonempty in $\Sigma$, it suffices to show that its boundary is empty. For this purpose, let us assume there exists a sequence $\mathsf z_n \in \C$ such that $\sigma_0(\mathsf z_n) \to \pmb p \in \partial \sigma_0(\C)$. As the image of the Gauss map $f_0':\C\to\Hyp^2$ of $\sigma_0$ is an ideal triangle, we can extract a subsequence (still denoted $\mathsf z_n$) such that $f_0'(\mathsf z_n)$ converge either to an ideal vertex or to a point on an edge of the ideal triangle. We will rule out both possibilities and thus get a contradiction. 

Let us first suppose that $f_0'(\mathsf z_n)$ converge to an ideal vertex. Since support planes converge to support planes, it follows that $\Sigma$ admits a lightlike support plane
$P$ (as a limit of the spacelike support planes to $\sigma_0(\C)$ at $\sigma_0(\mathsf z_n)$), which must be parallel to one of the three null directions in the closure of the Gauss map image. Now using Lemma \ref{lemma symmetry Hsurface}, let $Q$ be the timelike plane of symmetry of $\Sigma$ such that reflection in $Q$ leaves $P$ invariant. 

By Corollary \ref{cor proper axis}, the intersection of $\sigma_0(\C)$ with $Q$ is a proper spacelike curve. In particular, $\Sigma$ can contain no point in $Q \cap P$. In particular, the point $\pmb p$ is not in $Q$. Now let $\pmb p' \in \Sigma$ be the reflection of $\pmb p$ across $Q$. Since the reflection leaves $P$ invariant, $P$ must also be the support plane of $\Sigma$ at $\pmb p'$. Hence the midpoint of $\pmb p$ and $\pmb p'$ lies on the plane $Q$ and still has null support plane $P$, which contradicts the fact that the intersection of $\sigma_0(\C)$ with $Q$ is a proper spacelike curve

We are thus left with the case that $f_0'(\mathsf z_n)$ converge to an a point of an edge of the ideal triangle which is the image of the Gauss map. Let us now consider the new immersion 
$\sigma_1:\C\to\R^{2,1}$ defined by:
\begin{equation} \label{eq normal immersion}
\sigma_1(\mathsf z)=\sigma_0(\mathsf z)+f_0'(\mathsf z)~,
\end{equation}
where we are considering $f_0'$ as a map valued in $\Hyp^2\subset\R^{2,1}$. (Since $f_0'=\pmb n$ is the Gauss map of $\sigma_0$, the immersion $\sigma_1$ is the normal evolution of $\sigma_0$ at time 1.) By a direct computation, one obtains
$$d\sigma_1(v)=d\sigma_0(v)+D_{v}\pmb  n=d\sigma_0((\mathbbm 1+B)(v))~,$$
where $B$ is the shape operator of $\sigma_0$, and therefore the first fundamental form of $\sigma_1$ equals:
$$\langle d\sigma_1(v),d\sigma_1(w)\rangle=\I(\mathbbm 1+B) v,(\mathbbm 1+B) w)~.$$

By a direct computation from Equations \eqref{eq first fund form w}, \eqref{eq third fund form w} and \eqref{eq shape op w}, this metric has the form 
$4e^{2\tilde{h}_0}|d\mathsf w|^2$
which is a complete metric on $\C$ as $\tilde h_0>0$ (see Lemma \ref{lemma a priori1 bochner}). Therefore $\sigma_1(\C)$ is a proper immersion (Remark \ref{rmk complete implies entire}).
Now, since $\sigma_0(\mathsf z_n)$ is converging to $\pmb p$, and $f_0'(\mathsf z_n)$ is converging to some interior point of $\Hyp^2$, the sequence $\sigma_1(\mathsf z_n)$ is converging in $\R^{2,1}$ by Equation \eqref{eq normal immersion}.
But the sequence $\mathsf z_n$ escapes from compact sets of $\C$, hence $\sigma_1(\mathsf z_n)$ is diverging in $\R^{2,1}$, and this gives a contradiction.



In conclusion, following Lemma \ref{lem achronal extension}, $\Sigma = \sigma_0(\C)$, so in particular $\sigma_0(\C)$ is entire. Then by Corollary \ref{cor domains dependence} its domain of dependence must be the intersection of the futures of a set of null planes. Since the image of the Gauss map is an ideal triangle, Theorem \ref{thm gauss image} implies that this set must contain exactly three null planes. Since all triangular regular domains are equivalent up to isometry of $\R^{2,1}$, this concludes the proof.
 \qedhere
\end{proof}

\begin{remark} 
It actually turns out, as mentioned in Remark \ref{rmk CMC}, that $\sigma_1$ is a constant mean curvature embedding. The completeness of the first fundamental form of $\sigma_1$ is therefore also a consequence of more general results. However, the existence of such constant mean curvature surface only allows us to prove the second part of Proposition \ref{prop foliation triangular domain}, namely, to show that there is no converging sequence $\sigma_0(\mathsf z_n)$ for which the Gauss map converge to an interior point of $\Hyp^2$. Tools from CMC surfaces are not helpful to tackle the first point, that is, excluding the existence of lightlike support planes for $\sigma_0(\C)$.

The reason why such strategy fails is that, starting from an entire CMC surface, one can follow the normal flow in the past to obtain a CGC immersion, but it is hard to prove that it is entire (in fact, it  will  not be complete in general). Hence we do not adopt the language of CMC surfaces here, and the technical estimates on the holomorphic energy, leading to Corollary \ref{cor proper axis}, are the essential ingredient for our proof.
\end{remark}

In order to apply the entire CGC-$K$ surfaces provided by Proposition \ref{prop foliation triangular domain} as barriers for the general case, we will need to translate Proposition \ref{prop foliation triangular domain} in terms of Monge-Amp\`{e}re equations. In fact, fix three distinct points $\xi_1,\xi_2,\xi_3$ and three values $v_1,v_2,v_3$. 

Let $\Sigma^K$ be the surface provided by Proposition \ref{prop foliation triangular domain} with $\ph(\xi_i)=v_i$. Then the support function $u_{\Sigma^K}$ satisfies the Monge-Amp\`{e}re equation \eqref{eq monge constant curvature}. Moreover, by Theorem \ref{prop constant curvature convex hull2}, $u_{\Sigma^K}$ is necessarily affine on each side of the triangle $T$ having vertices $\xi_1,\xi_2,\xi_3$. Finally, the Legendre transform of $u_{\Sigma^K}$ gives the surface $\Sigma^K$, since $u_{\Sigma^K}$ is convex and lower semicontinuous. 

Hence we can reformulate our result in terms of Monge-Amp\`{e}re equations:

\begin{cor} \label{cor solution on triangle}
Given three distinct points $\xi_1,\xi_2,\xi_3\in\partial\D$, let $T$ be the triangle in $\overline \D$ with vertices $\xi_1,\xi_2,\xi_3$. For any affine function $l:\D\to\R$, there exists a solution $u$ to the problem
$$\begin{cases} 
\det D^2 u(\mathsf z)=\frac{1}{K}(1-|\mathsf z|^2)^{-2} & \text{for every }\mathsf z\in\mathrm{int}(T) \\
u|_{\partial T}=l|_{\partial T}~,
\end{cases}$$
such that the graph of the Legendre transformation of $u$ is an entire surface.
\end{cor}

\section{Existence and uniqueness for the Minkowski problem} \label{sec existence and uniqueness}

In this section we will prove the main result (Theorem \ref{thm mink problem}) concerning the Minkowski problem, namely the existence and uniqueness of entire surfaces of prescribed Gaussian curvature in any regular domain $\mathcal D_\ph$, where $\ph$ is finite on at least three points of $\partial\D$. 

\subsection{Statement of the Monge-Amp\`ere problem}
We start by constructing solutions to the Monge-Amp\`ere equation \eqref{eq monge ampere curvature}. Recall that the convex envelope of a function $\ph:\partial\D\to\R\cup\{+\infty\}$, introduced in Definition \ref{defi convex envelope}, is:
$$\mathrm{conv}(\ph)(\mathsf z)=\sup\{f(\mathsf z)\,|\,f:\overline \D\to\R\text{ is affine, }f|_{\partial\D}\leq\ph\}~.$$
Moreover, we denote by $\Omega_\ph\subseteq \D$ the interior of the convex hull of $\{\xi\in\partial\D\,|\,\ph(\xi)<+\infty\}$. 

\begin{theorem} \label{thm existence monge ampere}
Let $\ph:\partial\D\to\R\cup\{+\infty\}$ be any lower semi-continuous function which is finite on at least three distinct points. Let  $\psi\in C^\infty(\Omega_\ph)$ such that $a<\psi(\mathsf z)<b$ for some $a,b>0$ and for every $\mathsf z\in \Omega_\ph$. Then there exists a unique closed convex function $u: \overline{\D} \to \R \cup \{+\infty\}$ which is a solution to
\begin{equation} \label{eq dir supp}
\begin{cases}
\det D^2 u(\mathsf z)=\frac{1}{\psi(\mathsf z)}(1-|\mathsf z|^2)^{-2} \quad &\textrm{for $\mathsf z \in \Omega_\ph$} \\
u(\mathsf z) = \mathrm{conv}(\ph)(\mathsf z) \quad &\textrm{for $\mathsf z \in \overline{\D} \setminus \Omega_\ph$}~.
\end{cases}
\end{equation}
Moreover, $u$ is smooth on $\Omega_\ph$ and gradient-surjective.

\end{theorem}


Recall by Proposition \ref{prop reg domain conv} that $\mathrm{conv}(\ph)$ is equal to $\ph$ on $\partial \D$ and on any chord $[\xi_1,\xi_2]$ of $\partial \Omega_\ph$ it is the unique affine function interpolating $\ph(\xi_1)$ and $\ph(\xi_2)$. 

\subsection{Proof of existence}
We will split the proof of Theorem \ref{thm existence monge ampere} in two parts, by proving first the existence and then the uniqueness.

\begin{proof}[Proof of the existence part of Theorem \ref{thm existence monge ampere}]

Let us split the proof into several steps.
\begin{steps}
\item To simplify notation, we will write $h = \mathrm{conv}(\ph)$. Let us first construct an approximating sequence $u_n$.
Let $\Omega_n$ be an exhaustion of $\Omega_\ph$ by  strictly convex domains with smooth boundary, satisfying $\Omega_n\subset\!\subset \Omega_\ph$ and $\Omega_n\subset\!\subset\Omega_{n+1}$ for every $n$.
By the classical theory of Monge-Amp\`{e}re equations (Theorem \ref{thm existence classical monge ampere}), there exists a solution $u_n:\Omega_n\to\R$ of the problem 
$$\begin{cases}
\det D^2 u_n(\mathsf z)=\frac{1}{\psi(\mathsf z)}(1-|\mathsf z|^2)^{-2} & \text{for every }\mathsf z\in\Omega_n \\
u_n|_{\partial\Omega_n}=h|_{\partial \Omega_n} & 
\end{cases}
$$
which is continuous in $\overline \Omega_n$.
By Theorem \ref{solution strictly convex dimension 2}, $u_n$ is strictly convex and therefore smooth by Theorem \ref{solution smooth}.
\item 
We now need to prove some uniform \emph{a priori} estimates on the $u_n$. We claim that:
\begin{equation} \label{eq estimate monge ampere}
h(\mathsf z)-\frac{1}{\sqrt a}\sqrt{1-|\mathsf z|^2}\leq u_n(\mathsf z)\leq h(\mathsf z)
\end{equation}
for every $\mathsf z\in\Omega_n$. Indeed, the inequality $u_n\leq h$ follows from the fact that $h$ is the convex envelope of $\ph$ and $u_n$ is convex. For the other inequality, for every linear function $l$ on the disk, the comparison principle (Corollary \ref{cor comparison principle}) gives
$$l(\mathsf z) - \frac{1}{\sqrt a}\sqrt{1-|\mathsf z|^2} \leq u_n(\mathsf z)~.$$
Taking the supremum over all linear functions less than or equal to $\ph$ (compare with the proof of Proposition \ref{prop conv comparison}), we conclude $$h(\mathsf z)-\frac{1}{\sqrt a}\sqrt{1-|\mathsf z|^2}\leq u_n(\mathsf z)~.$$

\item We can now produce the solution $u_\infty$ as a limit of the $u_n$. In fact, 
it follows from the previous step that the functions $u_n$ are uniformly bounded on $\Omega_{n_0}$ for $n\geq n_0$. Moreover, since the $u_n$ are convex, by a classical argument they are equicontinuous on $\Omega_{n_0}$ for $n\geq n_0+1$, where the coefficient of equicontinuity depends on the uniform bound on $\Omega_{n_0+1}$ and on the distance between $\Omega_{n_0}$ and $\Omega_{n_0+1}$. 

Thus by the Ascoli-Arzel\`{a} theorem and a standard diagonal argument, we can extract a subsequence which converges uniformly on compact sets of $\Omega_\ph$ to a function
$u_\infty:\Omega_\ph\to\R$. By Lemma \ref{convergence of solutions}, $u_\infty$ satisfies
$$\det D^2 u_\infty(\mathsf z)=\frac{1}{\psi(\mathsf z)}(1-|\mathsf z|^2)^{-2}~.$$
Hence we have again that $u_\infty$ is strictly convex (Theorem \ref{solution strictly convex dimension 2}) and therefore smooth (Theorem \ref{solution smooth}).  

\item  Now define the closed convex function $u$ by $u = \mathrm{conv}(u_\infty)$. As $u_\infty$ is already convex, we have that $u$ coincides with $u_\infty$ on $\Omega_\ph$. It remains to show that $u = h$ on $\overline{\D} \setminus \Omega_\ph$. Both are infinite away from $\overline{\Omega}_\ph$, so we restrict attention to $\partial \Omega_\ph$. Let us first show that $u = h$ on $\overline{\Omega}_\ph \cap \partial\D$. Using \eqref{eq estimate monge ampere}, one obtains 
\begin{equation} \label{eq estimate monge ampere limit function}
h(\mathsf z)-\frac{1}{\sqrt a}\sqrt{1-|\mathsf z|^2}\leq u_\infty(\mathsf z)\leq h(\mathsf z)~.
\end{equation}

Taking the convex hull preserves these inequalities. Since $\frac{1}{\sqrt a}\sqrt{1-|\mathsf z|^2}$ vanishes on $\partial \D$, we conclude immediately that $u = h = \ph$ on $\partial \D$.

\item We are left with showing $u = h$ on $\partial \Omega_\ph \cap \D$. Let $c=[\xi_1,\xi_2]$ be any chord in $\partial\Omega_\ph$. Let $D_c$ be the half-plane in $\D$ bounded by $[\xi_1,\xi_2]$ which intersects $\Omega_\ph$ and let $u_c^a$ be the unique solution of the problem
$$\begin{cases}
\det D^2 u_c^a(\mathsf z)=\frac{1}{\sqrt a}(1-|\mathsf z|^2)^{-2} & \text{for every }\mathsf z\in D_c \\
u_c^a|_{\partial D_c}=0
\end{cases}~.
$$
If $c$ is the geodesic $\{x=0\}$ and $D_c=\D_+=\{x\geq 0\}$, then such solution was provided explicitly in Equation (\ref{eq chord barrier}), as it is the support function of a surface of revolution $\Sigma_0^a$, and by inspection it is continuous in $\overline{\D}_+$. In general, $u_c^a$ is the support function of a surface obtained by applying a linear isometry in $\SO_0(2,1)$ to $\Sigma_0^a$. Therefore also the solution $u_c^a$ is continuous on $\overline{D}_c$. 

By an argument similar to above (Equation \eqref{eq estimate monge ampere}) we get 
\begin{equation} \label{eq estimate monge ampere 2}
h(\mathsf z)+u_c^a(\mathsf z)\leq u_n(\mathsf z)\leq h(\mathsf z)~.
\end{equation}
By passing to the limit, we thus obtain
$$h(\mathsf z)+u_c^a(\mathsf z)\leq u_\infty(\mathsf z)\leq h(\mathsf z)~.$$
Taking convex envelopes as above and using that $u^a_c$ vanishes on $c$, we conclude that $u = h$ on $c$.\qedhere
\end{steps}
\end{proof}

\subsection{Proof of uniqueness}
Let us now complete the proof by showing the uniqueness of the solution $u$.

\begin{proof}[Proof of the uniqueness part of Theorem \ref{thm existence monge ampere}] 
We split the proof into two steps.
\begin{steps}
\item
Let $u$ now be \emph{any} solution of (\ref{eq dir supp}). As above, set $h = \mathrm{conv}(\ph)$. First we show that $u$ must satisfy the inequalities:
\begin{equation} \label{eq inequality for uniqueness 1}
h(\mathsf z)-\frac{1}{\sqrt a}\sqrt{1-|\mathsf z|^2}\leq u(\mathsf z)\leq h(\mathsf z)~.
\end{equation}
and, for every chord $c$ in $\partial\Omega_\ph$,
\begin{equation} \label{eq inequality for uniqueness 2}
h(\mathsf z)+u_c^a(\mathsf z)\leq u(\mathsf z)~,
\end{equation}
where $u_c^a$ is the solution, defined on the half-plane $D_c$ bounded by $c$, of:
$$\begin{cases}
\det D^2 u_c^a(\mathsf z)=\frac{1}{a}(1-|\mathsf z|^2)^{-2} & \text{for every }\mathsf z\in D_c \\
u_c^a|_{\partial D_c}=0
\end{cases}~.
$$

The inequality $u\leq h$ is obvious. The other inequality follows from an adaptation of the argument in \cite[Proposition 3.9]{Bonsante:2015vi}, where more details can be found. First, up to composing with an isometry, suppose $0\in\D$ is in $\Omega_\ph$. Fix $r\in(0,1)$ and let $u_r:\overline{\Omega}_\ph\to\R$ be defined by $u_r(\mathsf z)=u(r\mathsf z)$. It is then easy to check that
$$\det D^2 u_r(\mathsf z)\leq \frac{1}{a}(1-|\mathsf z|^2)^{-2}~.$$
Let us now define
$$h_r=\mathrm{conv}(u_r|_{\partial \Omega_\ph})~.$$

Since $u_r$ is continuous up to the boundary of $\Omega_\ph$ and $(u_r)|_{\partial\Omega_\ph}=(h_r)|_{\partial\Omega_\ph}$, again by the comparison principle we get
\begin{equation} \label{eq inequality for ur}
h_r(\mathsf z)+v(\mathsf z)\leq u_r(\mathsf z)~,
\end{equation}
where $v\leq 0$ can be any of the functions $(1/\sqrt a)\sqrt{1-|\mathsf z|^2}$ or $u_c^a$, for every chord $c$. 


It then turns out that 
\begin{equation} \label{eq liminf convex envelopes}
h(\xi)\leq\liminf_{r\to 1} h_r(\xi)
\end{equation}
for every $\xi\in \partial\Omega_\ph$. In order to show this, let $f$ be an affine function on $\D$ such that $f < h$. Since $u$ is lower-semicontinuous, the sublevel set $\{\mathsf z\in\overline\D\,:\,u(\mathsf z)\leq f(\mathsf z)\}$ is compact. Since $f$ is finite everywhere, it is contained in $\Omega_\ph$ and since it is compact it is contained in $r_0\Omega_\ph$ for some $r_0<1$.  

This shows that
$$u_r(\xi)=u(r\xi)>f(r\xi) =: f_r(\xi)$$
for every $\xi\in\partial\Omega_\ph$ and every $r\geq r_0$.
That is, $(f_r)|_{\partial\Omega_\ph}<(u_r)|_{\partial\Omega_\ph}$, which implies $f_r\leq h_r$. Taking the limit as $r\to 1$, this implies
$$f(\mathsf z)\leq \liminf_{r\to 1}h_r(\mathsf z)~.$$
The inequality \eqref{eq liminf convex envelopes} then follows, as $h(\xi)$ is defined as the supremum of $f(\xi)$ over all such affine functions $f$. Finally, taking the limit as $r\to 1$ in \eqref{eq inequality for ur}, we conclude the proof of \eqref{eq inequality for uniqueness 1} and \eqref{eq inequality for uniqueness 2}.

\item
Let $u_1$ and $u_2$ be any two solutions of (\ref{eq dir supp}). Using the inequality \eqref{eq inequality for uniqueness 1}, we obtain
$$-\frac{1}{\sqrt a}\sqrt{1-|\mathsf z|^2}\leq u_1(\mathsf z)-u_2(\mathsf z)\leq \frac{1}{\sqrt a}\sqrt{1-|\mathsf z|^2}~,$$
and similarly from \eqref{eq inequality for uniqueness 2},
$$u_c^a(\mathsf z)\leq u_1(\mathsf z)-u_2(\mathsf z)\leq -u_c^a(\mathsf z)~.$$
As already pointed out, $u_c^a$ is continuous up to $c$, on which it is zero. This implies that $u_1-u_2$ extends continuously to zero on $\partial \Omega_\ph$ and therefore by the comparison principle of Theorem \ref{comparison principle}, we have $\min(u_1-u_2)=0$. By reversing the roles of $u_1$ and $u_2$, we have $\min(u_2-u_1)=0$ and thus $u_1=u_2$.\qedhere

\end{steps} 
\end{proof}


 

\subsection{Proof of entireness} \label{sec entireness}

In this subsection, we prove that the solution $u$ obtained in Theorem \ref{thm existence monge ampere} is the support function of an entire surface, which is equivalent to the statement that $u$ is gradient surjective. This is the key step to conclude the existence of entire surfaces solving the Minkowski problem, hence in particular the classification of entire CGC-$K$ surfaces in $\R^{2,1}$.

\begin{proof}[Proof of the entireness part of Theorem \ref{thm existence monge ampere}]

The graph of the convex dual of $u$ is an achronal surface $\Sigma$ as in Lemma \ref{lem achronal extension}. As long as no point of $\Sigma$ has a support plane whose slope lies outside $\Omega_\ph$, we can recover all of $\Sigma$ as the graph of the Legendre transform of $u$ on $\Omega_\ph$, and so $\Sigma$ must be the spacelike CGC surface that we are looking for.

Suppose for the sake of contradiction that $\Sigma$ contains a point $\pmb p$ at which it admits a support plane $P$ whose slope lies outside of $\Omega_\ph$. By convex duality, the plane $P$ must be one of the planes in the boundary of $\Omega_\ph$ at which $u$ is finite: 
\begin{equation} \label{eq defi support plane contradiction}
P=\{\pmb x\in\R^{2,1}\,:\,\langle \pmb x,(\mathsf z,1)\rangle=u(\mathsf z)\}
\end{equation}
for some $\xi \in \partial \Omega_\ph$. If $\mathsf z$ is on the boundary of $\D$, set $\xi_1=\mathsf z$ and choose two other points $\xi_2,\xi_3$ such that $\ph(\xi_i)<+\infty$. Note that this is possible by the assumption that $\ph$ is finite at at least 3 points. If $\mathsf z$ lies on a chord $[\xi_1,\xi_2]$ of $\partial \Omega_\ph$, choose one other point $\xi_3$ at which $\ph(\xi_3) < + \infty$. Let us call $T_0$ the ideal triangle with vertices $\xi_1,\xi_2,\xi_3$. 

Let $l$ be the unique affine function with the property that $l(\xi_i) = \ph(\xi_i)$ for $i = 1,2,3$. By Proposition \ref{prop foliation triangular domain}, there exists an entire $K$-surface $\Sigma_0$, for $K=b$, with support function $u_0:\overline{\D}\to\R \cup \{+\infty\}$ satisfying $u_0|_{\partial T_0}=l|_{\partial T_0}$ and $u_0(\mathsf z) = +\infty$ for $\mathsf z \in \overline{\D} \setminus \overline{T_0}$.

Now, the functions $u$ and $u_0$ satisfy:
$$\det D^2 u=\frac{1}{\psi(\mathsf z)}(1-|\mathsf z|^2)^{-2}\leq \frac{1}{b}(1-|\mathsf z|^2)^{-2}= \det D^2 u_0~.$$
Moreover $u(\xi_i)=u_0(\xi_i)$ for $i=1,2,3$. Since $T_0$ is a polygon, the restriction of $u$ to $T_0$ is continuous by \cite[Theorem 2]{MR0230219}. Since $u$ is convex, we have $u|_{\partial T_0}\leq u_0|_{\partial T_0}$. Hence by the comparison principle (Corollary \ref{cor comparison principle}), $u|_{T_0}\leq u_0|_{T_0}$. This shows that $\Sigma$ is contained in the future of $\Sigma_0$. But the surface $\Sigma_0$ is entire and hence completely in the future of the support plane $P$ defined in Equation \eqref{eq defi support plane contradiction}.
This contradicts the assumption $\pmb p \in \Sigma \cap P$.
\end{proof}

\subsection{Conclusion of Minkowski and CGC problem}

We can now apply all the proved results and state the main theorems concerning the Minkowski problem and the CGC problem. In fact, in Theorem \ref{thm existence monge ampere} we construct a solution of the problem 
$$\begin{cases}
\det D^2 u(\mathsf z)=\frac{1}{\psi(\mathsf z)}(1-|\mathsf z|^2)^{-2} \quad &\textrm{for $\mathsf z \in \Omega_\ph$} \\
u(\mathsf z) = \mathrm{conv}(\ph)(\mathsf z) \quad &\textrm{for $\mathsf z \in \overline{\D} \setminus \Omega_\ph$}~.
\end{cases}$$
for every lower semicontinuous function $\ph$ finite on at least three points and every smooth bounded function $\psi$ defined on $\Omega_\ph$. Using this result, we now derive the complete solution to the Minkowski and CGC problems for entire spacelike surfaces.

\begin{reptheorem}{thm mink problem}
Given any regular domain $\mathcal D$ in $\R^{2,1}$ {which is not a wedge} and any smooth function $\psi$ defined on the image of the{ generalized} Gauss map of $\partial\mathcal D$ such that $a<\psi<b$ for some $a,b>0$, there exists a {unique} entire spacelike surface $\Sigma$ in $\mathcal D$ whose domain of dependence is $\mathcal D$ and whose curvature function satisfies:
$$\kappa(\pmb p)=\psi\circ G_\Sigma(\pmb p)~,$$
for every $\pmb p\in\Sigma$, where $G_\Sigma$ is the Gauss map of $\Sigma$.
\end{reptheorem}

\begin{proof}[Proof of Theorem \ref{thm mink problem}]

The existence part follows from Theorem \ref{thm existence monge ampere}. Indeed, by Proposition \ref{prop:bijection reg dom} the regular domain $\mathcal D$ must be equal to $\mathcal D_\ph$ for some closed function $\ph$ which is finite at at least 3 points. Then Theorem \ref{thm existence monge ampere} produces a function $u$ whose Legendre transform has the required properties.

Uniqueness is a straightforward consequence of Theorem \ref{thm existence monge ampere} together with Theorem \ref{prop constant curvature convex hull2}. Indeed, by Theorem \ref{prop constant curvature convex hull1}, for any surface $\Sigma$ satisfying the condition of Theorem \ref{thm mink problem}, the image of its Gauss map projected to the Klein model $\D$ of $\Hyp^2$, must be $\Omega_\ph$, i.e. the interior of the convex hull of $\{\xi\in\partial\D\,:\,\ph(\xi)<+\infty\}$. Moreover, the two bullet points of Theorem \ref{prop constant curvature convex hull2} imply that its support function $u_\Sigma$ agrees with $\mathrm{conv}(\ph)$ on $\overline{\D}\setminus \Omega_\ph$. Therefore, $u_\Sigma$ must be a solution of the problem \eqref{eq dir supp} of Theorem \ref{thm existence monge ampere}, and by the uniqueness part of Theorem \ref{thm existence monge ampere}, we conclude that $\Sigma$ is unique.
\end{proof}

As a particular case, we therefore obtain the solution to the CGC problem in regular domains different from a wedge:

\begin{reptheorem}{thm-CGC-problem}
Fix $K>0$. Given any regular domain $\mathcal D\subset \R^{2,1}$ which is not a wedge, there exists a unique entire CGC $K$-surface whose domain of dependence is $\mathcal D$.
\end{reptheorem}

We conclude with the following classification result for entire CGC-$K$ surfaces.

\begin{repcor}{cor classification Ksurfaces}
Fix $K>0$. There is a bijection between the set of future-convex entire surfaces of constant Gaussian curvature $K$ in $\R^{2,1}$ and the set of lower semicontinuous functions $\ph:\partial\D\to\R\cup\{+\infty\}$ finite on at least three points, which is defined by $\Sigma\mapsto (u_\Sigma)|_{\partial\D}$.
\end{repcor}

\begin{proof}
Let $\Sigma$ be a future-convex entire CGC-$K$ surface. By Corollary \ref{cor domains dependence}, the domain of dependence of $\Sigma$ is of the form $\mathcal D_\ph$ for some lower semi-continuous function $\ph$. In light of Theorem \ref{thm-CGC-problem}, it remains only to rule out the possibility that $\ph$ is finite at 0, 1, or 2 points.

But by Theorem \ref{thm gauss image}, the image of the Gauss map of $\Sigma$ must be the interior of the convex hull of those points where $\ph$ is finite.
Clearly the image of the Gauss map of $\Sigma$ must be nonempty, which rules out the cases of 0 or 1 points, and by strict convexity it must also have interior, which rules out the case of 2 points.
\end{proof}

As a remark, we mention that there is a natural action of $\Isom(\R^{2,1})=\SO_0(2,1)\rtimes\R^{2,1}$ on the set  of future-convex entire surfaces of constant Gaussian curvature. Under the bijection of Corollary \ref{cor classification Ksurfaces}, this action corresponds to a natural action of the semi-direct product $\SO_0(2,1)\rtimes\R^{2,1}$ on the set of lower semi-continuous functions. See also \cite{seppisurvey}.


\section{Foliations by CGC surfaces}\label{sec:foliation}

The purpose of this section is to show that any regular domain $\mathcal D$ which is not a wedge if foliated by the (unique) CGC-$K$ surfaces $\Sigma_K$ having domain of dependence $\mathcal D$, as for Theorem \ref{thm-CGC-problem}.

\begin{reptheorem}{thm CGC foliation}
For every regular domain $\mathcal D$ in $\R^{2,1}$ which is not a wedge, there exists a unique foliation by properly embedded CGC-$K$ surfaces, as $K\in(0,\infty)$.
\end{reptheorem}

\begin{proof}
There are three main steps in the proof. The first two steps will show that the $K$-surfaces provide a foliation of a region of $\mathcal D$. Then the third step will show that this region ``fills up'' $\mathcal D$ close to $\partial\mathcal D$ and close to infinity.
\begin{steps}
\item Let us first show that the CGC-$K$ surfaces are disjoint: more precisely, if $K_1<K_2$, then $\Sigma_{K_1}$ is in the future of $\Sigma_{K_2}$. Let $u_1$ and $u_2$ be the support functions of $\Sigma_{K_1}$ and $\Sigma_{K_2}$ respectively. Both $u_1$ and $u_2$ satisfy the inequalities \eqref{eq inequality for uniqueness 1} and \eqref{eq inequality for uniqueness 2} for $K = K_1$. Therefore, their difference tends to zero at the boundary of $\Omega_\ph$. By the comparison principle, $u_2-u_1$ cannot have an interior minimum, and so it follows that $u_2 > u_1$ strictly on the whole domain $\Omega_\ph$. From the formula for the convex dual, we see immediately that $\Sigma_{K_1}$ lies weakly in the future of $\Sigma_{K_2}$; since both surfaces are entire, the strict inequality $u_2 > u_1$ implies that they cannot be tangent at any point, and so we in fact have that $\Sigma_{K_1}$ lies strictly in the future of $\Sigma_{K_2}$.

\item It remains to show that, for every point $\pmb x\in\mathcal D$, there exists a CGC surface $\Sigma_K$ such that $\pmb x\in \Sigma_K$. Hence let 
$$\mathcal I_-:=\{K\in(0,\infty)\,:\,\pmb x\in \overline{I^-(\Sigma_K)}\}~,$$
and analogously 
$$\mathcal I_+:=\{K\in(0,\infty)\,:\,\pmb x\in \overline{I^+(\Sigma_K)}\}~.$$
We emphasize that $\mathcal I_-$ corresponds to surfaces lying above $\pmb x$ and $\mathcal I_+$ corresponds to surfaces lying below $\pmb x$. In this step we show that $\mathcal I_-$ and $\mathcal I_+$ are nonempty. 

Let $\ph = \ph_{\mathcal D}$. The point $\pmb x$ corresponds to an affine plane $P_{\pmb x}$ in $\D \times \R$ which lies strictly below the graph of $\ph$. Moreover, the point $\pmb x$ is in $I^+(\Sigma_K)$ if and only if the plane $P_{\pmb x}$ lies below the graph of the support function $u_{\Sigma_K}$. By \eqref{eq inequality for uniqueness 1} we have for all $K$
$$\mathrm{conv}(\ph)(\mathsf z)-\frac{1}{\sqrt K}\sqrt{1-|\mathsf z|^2}\leq u_{\Sigma_K}(\mathsf z)\leq \mathrm{conv}(\ph)(\mathsf z)~.$$
Therefore, as $K \to \infty$, the support function $u_{\Sigma_K}$ converges to $\mathrm{conv}(\ph)$, so for $K$ large enough, $P_{\pmb x}$ lies below the graph of $u_{\Sigma_K}$. This shows that $\mathcal I_+$ is nonempty.

On the other hand, let $\mathcal{D}_0$ be any triangular domain containing $\mathcal D$. Since it is invariant under rescaling, $\mathcal{D}_0$ is foliated by rescaled copies of its corresponding CGC-1 triangular surface $\Sigma^0$. As the rescaled surfaces tend towards infinity, their curvature tends to zero. In particular, for some $K$ small enough, $\Sigma^0_{K}$ is in the future of $\pmb x$. By the same application of the comparison principle as in the proof of entireness (Section \ref{sec entireness}), $\Sigma_K$ lies in the future of $\Sigma^0_{K}$. This shows that $\mathcal I_-$ is nonempty.

\item By the previous two steps, $\mathcal I_-$ and $\mathcal I_+$ are nonempty and connected, with $\inf \mathcal{I_-} = 0$ and $\sup \mathcal I_+ = + \infty$. In this step, we show that $\sup \mathcal I_- = \inf \mathcal I_+ =: K_{\pmb x}$ and $\pmb x \in \Sigma_{K_{\pmb x}}$.

Let $K_- = \sup \mathcal I_-$, and let $K_i \in \mathcal I_-$ be an increasing sequence tending to $K_-$. Then $\Sigma_{K_i}$ are a decreasing sequence of surfaces with $\pmb x$ in their past. The corresponding support functions $u_i$ form an increasing sequence bounded from above by $\mathrm{conv}(\ph)$, so they converge uniformly on compact subsets of $\Omega_\ph$ to a limit $u_\infty$. By Lemma \ref{convergence of solutions}, $u_\infty$ is a solution of the Monge-Amp\`{e}re equation \eqref{eq monge constant curvature} on $\Omega_\ph$ with curvature $K_-$. Passing to the limit, inequalities \eqref{eq inequality for uniqueness 1} and \eqref{eq inequality for uniqueness 2} applied to $u_i$ show that $u_\infty$ is equal to $\mathrm{conv}(\ph)$ on $\partial \Omega_\ph$. Hence by the uniqueness part of Theorem \ref{thm existence monge ampere}, $u_\infty$ is the support function of $\Sigma_{K_-}$.

Since the functions $u_i$ converge uniformly on compact sets to a strictly convex limit, the convergence must be at least $C^1$, and therefore the dual surface $\Sigma_{K_i}$ also converge locally uniformly to $\Sigma_{K_-}$. In particular, since $\pmb x \in I^-(\Sigma_{K_i})$, it follows that $\pmb x \in \overline{I^-(\Sigma_{K_-})}$. Therefore $K_- \in \mathcal I_-$.

Similarly, if $K_+ = \inf \mathcal I_+$, we can produce an increasing family of surfaces $\Sigma_{K_i}$, with a corresponding decreasing family of support functions $u_i$. Since $\pmb x$ is in the future of each surface $\Sigma_{K_i}$, the corresponding affine plane $P_{\pmb x}$ in $\D \times \R$ lies below the graph of each function $u_i$. Since the $u_i$ are bounded below, they again converge uniformly on compact sets and as above the limit is the support function of $\Sigma_{K_+}$. We conclude similarly that $K_+ \in \mathcal I_+$.

We have now shown that $\mathcal I_-$ and $\mathcal I_+$ are closed as subsets of $(0,\infty)$. They clearly cover the interval $(0,\infty)$, so they must have nonempty intersection. On the other hand, the intersection consists precisely of those $K$ such that $\pmb x \in \Sigma_K$, which by Step 1 must be a single value. Hence $K_- = K_+$, and this surface contains $\pmb x$.
\qedhere
\end{steps}
\end{proof}

\section{Open questions}\label{sec:final}

We conclude by mentioning an open question on the subject. Theorem \ref{thm-CGC-problem} provides a classification of entire CGC-$K$ surfaces in $\R^{2,1}$, which is also stated in Corollary \ref{cor classification Ksurfaces}. It would be interesting to classify \emph{complete} CGC-$K$ surfaces (which are automatically entire), in terms of the function $\ph=(u_\Sigma)|_{\partial\D}$. This would give a classification of all $C^2$ isometric immersion of the hyperbolic plane into $\R^{2,1}$.

In \cite{Bonsante:2015vi} a characterization of surfaces \emph{with bounded principal curvature} was obtained. That is, an entire surface $\Sigma$ has bounded principal curvatures (that is, the principal curvatures are in an interval $[1/C,C]$ for some $C>0$) if and only if $\ph$ has the Zygmund regularity. Let us observe that, if $\Sigma$ has bounded principal curvatures, then it is complete, since the Gauss map is bi-Lipschitz in this case. In particular, this characterization does not depend on $K$. 

Hence the class of complete CGC-$K$ surfaces correspond to a subset of the space of lower semicontinuous functions, finite on at least three points, which contains Zygmund regular functions. In particular, it contains Lipschitz functions. We actually have some negative examples: first, the entire CGC-$K$ surfaces considered in Section \ref{sec triangular surfaces}, are not complete, since the induced metric is homothetic to an ideal triangle in $\Hyp^2$. In this case, the support function is only finite on three points. Moreover, in \cite{Bonsante:2015vi} another example was provided, namely an entire non-complete surface (the induced metric is isometric to the universal cover of $\Hyp^2\setminus\{p\}$ for a point $p\in\Hyp^2$), whose support function on $\partial\D$ is:
$$\ph(\xi)=\begin{cases} a & \textrm{if }\xi=\xi_0 \\ b & \textrm{if }\xi\neq\xi_0 \end{cases}~,$$
for any $a<b$. 

We remark that in all example we know of noncomplete entire CGC-$K$ surfaces, the support function $\ph$ on $\partial \D$ has the property that there is a point $\xi_0 \in \partial \D$ at which
$$\liminf_{\substack{\xi\to\xi_0\\ \xi\neq \xi_0}}\ph(\xi)>\ph(\xi_0)~.$$
It would be interesting to know if this is a necessary or sufficient condition.





\addtocontents{toc}{\SkipTocEntry}
\bibliographystyle{alpha}
\bibliographystyle{ieeetr}
\bibliography{../bs-bibliography.bib}

\end{document}